\documentclass[12pt,twoside,a4paper]{amsart}
\usepackage{amssymb}
\usepackage{latexsym}
\usepackage{amsmath}
\usepackage{euscript}

\usepackage{hhline}
\usepackage{amssymb,latexsym,amsmath,amsthm}
\makeindex

   \def\sH{{\mathfrak H}}   
      
\def\sM{{\mathfrak M}}   \def\sN{{\mathfrak N}}   
      
\def\sS{{\mathfrak S}}

      \def\dC{{\mathbb C}}

      \def\dR{{\mathbb R}}

      \def\cC{{\mathcal C}}
   \def\cE{{\mathcal E}}   
   \def\cH{{\mathcal H}}

\newcommand{\fH}{\mathfrak H}

 \DeclareMathOperator{\ran}{ran}
\DeclareMathOperator{\dom}{dom} 
 
\DeclareMathOperator{\Ext}{Ext} 
\DeclareMathOperator{\supp}{supp} 
\DeclareMathOperator{\sign}{sign} 
\DeclareMathOperator{\shh}{sh}
\newtheorem {theorem}     {Theorem} [section]    
\newtheorem {definition} [theorem] {Definition} 

\newtheorem {lemma}    [theorem]   {Lemma}       
\newtheorem {corollary}[theorem]   {Corollary}   
\newtheorem {proposition}[theorem]{Proposition} 
\theoremstyle{remark}
\newtheorem {example}   [theorem] {Example} 
\newtheorem {remark}    [theorem] {Remark} 

\newtheorem*{remark*}{Remark}
\renewenvironment {proof} {\begin{trivlist} \item[\hspace{\labelsep}%
\sc Proof.]}{$\Box$ \end{trivlist}}

\newcommand {\wt}{\widetilde}
\newcommand {\wh}{\widehat}

\renewcommand {\Im}{\mathop{\rm Im}\nolimits}

\def\IM{{\rm Im\,}}
\def\ran{{\rm ran\,}}

\def\dom{{\rm dom\,}}

\def\dim{{\rm dim\,}}

\newsymbol\mr 1336
\newsymbol\br 133E

\newcommand {\R}{\mathbb{R}}

\newcommand{\eq}[2]{\begin{equation}{{#2}\label{#1}} \end{equation}}
\makeatletter \@addtoreset{equation}{section}\makeatother

\begin{document}
\title[Weyl function of a Hermitian operator]
{Weyl function of a Hermitian operator and its connection with characteristic function}

\author[V. Derkach, M. Malamud]{Vladimir Derkach, Mark Malamud}
\email{derkach.v@gmail.com}
\email{malamud3m@yahoo.com}

\maketitle

 \tableofcontents

\renewcommand{\contentsname}{Contents}

\section{Introduction}
Let $A$ be a closed symmetric operator with a dense domain $\dom(A)$  in a separable
Hilbert space $H$, and having equal deficiency indices  $(n,n)$ with $n\leq\infty$.
Characteristic operator functions of the operator  $A$, as well as characteristic
operator functions of its self-adjoint extensions, introduced originally as a subject in
seminal papers by M.S. Livsi\v{c} \cite{Liv46}-\cite{Liv50} and were studied afterwards
in numerous papers  (see for instance \cite{Arl80}-\cite{ArlTs74}, \cite{Brod69},
\cite{Koc80}, \cite{Kuz59}-\cite{Kuz80}, \cite{SNF}, \cite{ShmTs77}-\cite{KL73},
\cite{ShmTs77} and references therein).

In this paper we present a new approach to the concept  of the characteristic operator
function, which differs from that used in  the abovementioned papers and which seems to
be more simple and natural. It is based on the abstract version of the second Green
formula, formalized in the notion of the "abstract boundary value" \cite{Koc75},
\cite{Koc79}, \cite{Mikh80}, and is closed to the approach elaborated by A.N. Kochubej
in~\cite{Koc80}.

Let us briefly describe  the content of the paper. First we  remind  the following
definition (see~\cite{Koc75}).
\begin{definition}\label{D:DM_1}
Let $\mathcal{H}$ be a Hilbert space, and let $\Gamma_0$ and $\Gamma_1$ be linear mappings from
 $\dom(A^*)$ to $\mathcal{H}$. A {triplet} $\{\mathcal{H}, \Gamma_0, \Gamma_1\}$ is called a boundary {triplet} for the operator
 $A^*$, if:
 \begin{itemize}
    \item [(i)] the following abstract Green identity holds
    \begin{equation}\label{DM_1}
      (A^*f,g)_\fH-(f,A^*g)_\fH=(\Gamma_1f,\Gamma_0g)_{\cH}-(\Gamma_0f,\Gamma_1g)_{\cH},\qquad f,g\in\dom{A^*};
    \end{equation}
    \item [(ii)] and the mapping $\Gamma=\begin{pmatrix}
    \Gamma_1\\ \Gamma_0\end{pmatrix}:\;
    \dom A^*\rightarrow\cH\oplus\cH$ is surjective.
  \end{itemize}
\end{definition}
In the following definition we introduce a class of extensions of the operator  $A$,
which is quite useful in many questions, and, in particular, for our purposes. An
extension $\widetilde{A}$ of the operator $A$ is called {\it a proper extension}, if
$A\subset\widetilde{A}\subset A^*$.
\begin{definition}\label{D:DM_2}
A proper extension $\widetilde{A}$ of the operator $A$ will be called almost solvable (a.e.),
if there exists a boundary {triplet}  $\{\mathcal{H}, \Gamma_0, \Gamma_1\}$ and an operator  $B\in[\cH]$, such that
\begin{equation}\label{DM_2}
\dom(\widetilde{A} )=\{ f\in\dom(A^*):\, \Gamma_1f=B\Gamma_0f\}.
\end{equation}
A proper extension $\widetilde{A}$ of the operator $A$, determined by the
equality~\eqref{DM_2} will be denoted by ${A}_B$.
\end{definition}

The class of almost solvable extensions is big enough, as follows from the results of
Section~\ref{1}, were two criteria for an extension $\wt A$ to be  almost solvable are
presented. Let us notice that the class of almost solvable extensions of $A$ contains
proper extensions $\widetilde{A}$ of the operator $A$ with two regular points $z_1, z_2
\in \rho(\widetilde{A})$, such that $\IM z_1\cdot\IM z_2<0$  (see
Proposition~\ref{P:DM_5}), and  it contains all proper extensions $\widetilde{A}$ of $A$
whenever the defect index of $A$ is finite.

Let us also note that the scope of the method is not restricted to the class of almost
solvable extensions of symmetric operators only.  Certain results are obtained for
extensions determined by~\eqref{D:DM_2} with unbounded
 $B\in\mathcal{C}(\mathcal{H})$ (see Propositions~\ref{P:DM_7}-\ref{P:DM_9}, Theorems~\ref{T:DM_2}, \ref{T:DM_3}).

In Section~\ref{2}, influenced by the analogy with the Sturm-Liouville operator, we
associate to each boundary {triplet}  $\{\mathcal{H}, \Gamma_0, \Gamma_1\}$ an operator
valued function $M(z)$ via the equality
\[
M(z)\Gamma_0f_z=\Gamma_1f_z\qquad (f_z\in \sN_z,\ z\in \rho(A_0)).   
\]
It is shown that  $M(z)$ is holomorphic on $\rho(A_0)$  operator-valued function with
values in $[\mathcal{H}]$ where ${A}_0={A}_0^* = A^*\lceil \ker\Gamma_0$. In what
follows $M(z)$  will be called  the abstract Weyl function. We show
(Theorem~\ref{T:DM_1}), that it is a $Q$ -- function (in the sense of M.Krein and
H.Langer~\cite{KL71}) of the symmetric operator $A$, corresponding to the extension
${A}_0$.  In particular, it belongs to the class $(R_\cH)$
 \footnote[1]{Class $(R_\cH)$ consists of functions $F(z)$ holomorphic in the upper
 half-plane $\dC_+$, taking values in $[\cH]$
  and such that $\Im F(z)\ge 0$  for $z\in\dC_+$.
 All required  information on  the class $(R)$ can be found in~\cite{KrNud}.}
and determines the pair $\{A,{A}_0\}$ uniquely up to the unitary equivalence provided
that $A$ is simple. Therefore in the later case the spectrum of $A_0$ is implicitly
described by means of the Weyl function.

For every symmetric operator  $A$ the Weyl function $M(z)$ plays a role similar to that
of the classical Weyl function~\footnote[2]{Weyl function in the papers~\cite{Marc52},
\cite{Pavlov66} differs by the sign from the function $m_h(z)$ in~\cite{Naj69}} $m_h(z)$
(\cite{Marc52}, \cite{Pavlov66}) for the Sturm-Liouville operator on the half-line, and
coincides with the latter for an appropriate choice of the boundary {triplet}. Further,
if $A$ is a minimal differential operator of order $2n$ on the half-line with the
deficiency index $(n,n)$, then $M(z)$ coincides with the characteristic matrix
\cite{Naj69} of certain its extension ${A}_B={A}_B^*$. In the case of  a differential
operator  $A$  with operator-valued  coefficient considered either on a finite
interval~\cite{RB69}, or on a semi-axes~\cite{Khol'kin}, the operator-valued function
$M(z)$ coincides with the characteristic function introduced in~\cite{RB69}  and
\cite{Khol'kin}, respectively.

In the case of a nonnegative operator $A$, a criterion   for the Weyl function $M(z)$ to
belong to the Stieltjes class $(S)$ \cite{KrNud} is given (see Theorem~\ref{T:DM_3}).
Moreover, a  connection is found between the operator-valued function $M(z)$ on the one
hand, and the operator-valued functions $Q_{F}(z)$ и $Q_{M}(z)$ introduced
in~\cite{KrOv78} on the other hand. Note in this connection that $Q_{F}(z)$ и $Q_{M}(z)$
 are not properly $Q$-functions in the sense of~\cite{KL71}.

Let us notice that in the upper half plane the abstract Weyl function $M(z)$ is a
linear-fractional transform of the characteristic function of the symmetric operator $A$
introduced in \cite{Koc80}, which, up to non-essential details, turns out to be the
characteristic function~\cite{Str68} in the sense of A.V. Shtraus.

The main result of Section~\ref{3} states that for any boundary triplet  $\{\mathcal{H},
\Gamma_0, \Gamma_1\}$  for $A^*$,  any pair of proper extensions ${A}_{B_k}$ of $A$ with
$B_k\in\mathcal{C}(\mathcal{H})\  (k=0,1)$,   the following equivalences hold:
$$
({A}_{B_1}-z)^{-1}-({A}_{B_0}-z)^{-1}\in\mathfrak{S}(H)\;\Longleftrightarrow\;
(B_1-M(z))^{-1}-(B_0-M(z))^{-1}\in\mathfrak{S}(\mathcal{H})
$$
$$
\;\Longleftrightarrow\; (B_1-\xi)^{-1}-(B_0-\xi)^{-1}\in\mathfrak{S}(\mathcal{H}).
$$
%
%
 Here   $M(z)$ is the corresponding Weyl function,  $\mathfrak{S}(H)$ denotes a
two-sided ideal in $[H]$,  $z\in \rho({A}_{B_0}\cap \rho({A}_{B_1}),$  and
$\xi\in\rho(B_0)\cap \rho(B_1)$.
 As immediate consequences of this result we derive  certain statements
 (see Corollaries \ref{C:DM_7}-\ref{C:DM_11})  regarding  the asymptotic
behavior of the spectrum of the operator ${A}_B$.

In Section \ref{4} the negative spectrum of a self-adjoint extension ${A}_B={A}_B^*$
($B\in [\mathcal{H}]$) of a nonnegative operator $A\ge 0$ is investigated. It is shown
here that the dimensions  of the  "negative" spectral subspaces $E_{A_B}(-\infty, 0)$
and $E_{B-M(0)}(-\infty, 0)$  of the operators  ${A}_B$ and  $B-M(0)$, respectively,
coincide. This statement extends and generalizes the results from~\cite{Bir56},
\cite{Kr47}, \cite{Mikh80}, and  coincides  with them in the case of a uniformly
positive operator $A$ and a special choice of a boundary triplet (such a choice ensures
that $M(0)=0$  whenever $m_A>0$).

In Section \ref{5} characteristic functions of almost solvable  extensions of the
operator $A$ are studied. As is known \cite{GK82}, if a minimal Sturm-Liouville operator
$A$ on the half-line is in the limit point case at $\infty$, then the characteristic
function $\Theta(z)$ of its extension ${A}_h\quad({A}_hy = -y''+q(x)y, \;
y'(0)=hy(0),\;h\neq\overline{h})$ is connected with the classical Weyl function
$m_{\infty}(z)$ by a linear-fractional transformation
\[
\Theta(z)=(m_{\infty}(z)+h)(m_{\infty}(z)+\overline{h})^{-1}.
\]
In Theorem~\ref{T:DM_3} it is shown that for the characteristic function of almost
solvable  extension ${A}_B$ (see~\eqref{DM_2}) of   $A$ an analogous formula holds, with
$m_{\infty}(z)$ and $h$ replaced by $M(z)$ and the operator $B$ from~\eqref{DM_2},
respectively. Also, the inverse problem for characteristic function of a.s. extensions
is solved, i.e. a criterion for an analytic operator-valued function $\Theta(z)$ to be a
characteristic function of a.s. extension  ${A}_B$ of a symmetric operator $A$. Notice
also, that the proof of this result is essentially relied on the Kre\u{\i}n-Langer
construction from~\cite{KL73}.

In Sections \ref{6}-\ref{9} different differential operators are considered, for which
Weyl functions and characteristic functions are found and spectra of extensions are investigated.
Namely, ordinary differential operators are studied in Section~\ref{6},
 differential operators with unbounded operator coefficients are studied in Section~\ref{7},
 Shr\"{o}dinger operator in $\mathbb{R}^3\setminus
\{0\}$ and Laplacian operator in domains $\Omega\subset\mathbb{R}^2$ with non-smooth boundaries are
studied in Section~\ref{8} and Section~\ref{9}, respectively.
Let us note that while the boundary {triplet}s for operator $A^*$ from Sections~\ref{6}-\ref{8} were
found in~\cite{VG75}, \cite{Gor71}, \cite{Kut76}, \cite{RB69}, the boundary {triplet} for operator $A^*$
from Section~\ref{9} is constructed here for the first time.

In conclusion, let us emphasize  that
a boundary triplet plays a role of a "coordinate
system"  in analytic geometry.
It leads to a natural parametrization   of the proper  extensions of $A$ by means of
linear relations (multi-valued operators) in $\cH$.
An adequate  treatment  of certain spectral problem for extensions can be achieved by
using an appropriate  boundary triplet. 


In particular, the  A.V. Shtraus' approach \cite{Str60}, based on the J.von Neumann theory, is
equivalent, in essence, to a  choice of a "canonical" boundary {triplet}. However, as it
is clear  from numerous examples (see for instance sections \ref{6}-\ref{8}) to each
differential expression one associates  a natural boundary {triplet}, which in general,
is very far to be "canonical".  However, all the computations, connected with the
characteristic function, and description of spectral properties of extensions, become
more explicit and simpler, if the corresponding boundary {triplet}s are naturally
related to  the problem. Therefore, we see the main advantage of the proposed approach
in the flexibility  of a  choice of a boundary {triplet}.

In what follows  we use the following notations:
\begin{itemize}
\item $\mathbb{C}_+(\mathbb{C}_-)$ -- open upper (lower) half-plane;
    \item $H, \mathcal{H}$ --  Hilbert spaces;
    \item  $[\mathcal{H}_1, \mathcal{H}_2]$ -- the set of bounded linear operators from $\mathcal{H}_1$ to $\mathcal{H}_2$; if
$\mathcal{H}_1=\mathcal{H}_2=\mathcal{H}$, then $[\mathcal{H}_1,
\mathcal{H}_2]=[\mathcal{H}]$;
    \item $\mathcal{C}(\mathcal{H})$ -- the set of closed densely defined operators in  $\mathcal{H}$;
    \item $\mathfrak{S}(\cH)$ -- two-sided ideal in the ring $[\cH]$;
    \item $\sS_\infty(\cH)$ -- the set of compact operators in $\cH$;
    \item $s_j(B)=\lambda_j((B^*B)^{1/2})$ -- s-numbers of the operator $B\in\sS_\infty(\cH)$;
\item $B_I=\IM B=(B-B^*)/2i$ -- imaginary component of the operator $B\in[\cH]$.
    \item $A$ -- symmetric operator in $H$, $\dom(A)$ -- its domain,
    \item $\mathfrak{M}_z=(A-z)\dom(A)$,
$\mathfrak{N}_z=\mathfrak{M}_{\bar{z}}^{\bot}=\ker(A^*-z)$,
    \item $n_{\pm}(A)=\dim\mathfrak{N}_{\pm i}$ -- deficiency index of the operator $A$;
    \item $\widetilde{A}$ -- proper extension of the operator  $A$, i.e.
$A\subset\widetilde{A}\subset A^*$,
    \item $\sigma_p(\widetilde{A}), \sigma_c(\widetilde{A}),
\sigma_r(\widetilde{A})$ -- point spectrum, continuous spectrum and residual spectrum of the operator  $\widetilde{A}$;
\item $\rho(\widetilde{A})$,
$\widehat{\rho}(\widetilde{A})$ -- its resolvent set and the field of regularity of $A$, respectively;
\item $\sigma(\widetilde{A}):=\dC\setminus\rho(\wt A)$,
$\widehat{\sigma}(\widetilde{A}):=\dC\setminus\wh{\rho}(\wt A)$ --
spectrum and the core of the spectrum of the operator $\widetilde{A}$;
\item $R_{\widetilde{A}}(\lambda)=(\widetilde{A}-\lambda)^{-1}$ --
the resolvent of the operator  $\widetilde{A}$;
\item  $A_{F}, A_{K}$ ---  Friedrichs extension and Kre\u{\i}n extension \cite{Kr47}  of the nonnegative operator  $A$.
\end{itemize}


\section{Almost solvable  extensions}\label{1}
In this section we consider proper extensions $\widetilde{A}$ of a symmetric  operator
$A$ and find certain sufficient conditions and criteria for $\widetilde{A}$ to be almost
solvable.    Some useful properties of boundary {triplet}s for $A^*$ will also be
established.

With each boundary {triplet} $\{\cH, \Gamma_0, \Gamma_1\}$ for $A^*$ one  associates two
self-adjoint extensions ${A}_0$ and ${A}_1$ of the operator $A$ by setting:
\begin{equation}\label{DM_3}
\dom({A}_0)=\ker \Gamma_0,\quad \dom({A}_1)=\ker \Gamma_1.
\end{equation}
Clearly, the extensions ${A}_0$ and ${A}_1$  are transversal in the sense of the
following definition.
\begin{definition}\label{D:DM_3}
Two proper extensions $\widetilde{A}_0$ and $\widetilde{A}_1$ of an operator $A$ are
called  disjoint if
\begin{equation}\label{DM_4}
\dom(\widetilde{A}_0)\cap\dom(\widetilde{A}_1)=\dom(A),
\end{equation}
They are called transversal, if in addition
\begin{equation}\label{DM_5}
\dom(\widetilde{A}_0)+\dom(\widetilde{A}_1)=\dom(A^*).
\end{equation}
\end{definition}
The converse statement  is also true  (see Corollary~\ref{C:DM_1} of
Proposition~\ref{P:DM_2}).   Now we only notice that for each extension $\widetilde{A}_1
=\widetilde{A}_1^*$ of the operator $A$ there is an  extension
$\widetilde{A}_0=\widetilde{A}_0^*$ of $A$ transversal to $\widetilde{A}_1$ and a
"canonical" boundary {triplet} $\{\cH, \Gamma_0, \Gamma_1\}$, corresponding  to the pair
$(\widetilde{A}_0, \widetilde{A}_1)$ in the sense of  equalities~\eqref{DM_3} for
$A_0=\widetilde{A}_0$ and $A_1=\widetilde{A}_1$. Indeed, in accordance with the second
J. von Neumann formula
\[
\dom(\widetilde{A}_1)=\dom(A)\dotplus(I-V)\mathfrak{N}_i,\quad \mbox{where}\quad
V\in[\mathfrak{N}_i,\mathfrak{N}_{-i}] 
\]
is an isometry from $\mathfrak{N}_i$ onto  $\mathfrak{N}_{-i}$. The operator
$\widetilde{A}_0=\widetilde{A}_0^*$  defined  by
\[
\dom(\widetilde{A}_0)=\dom(A)\dotplus(I+V)\mathfrak{N}_i
\]
will be called  canonically transversal to the operator $\widetilde{A}_1$. Setting
  \begin{equation}\label{DM_6}
\Gamma_0^0=-P_{-i}+VP_i, \quad \Gamma_1^0=iP_{-i}+iVP_i, \quad \cH^0=\mathfrak{N}_{-i},
  \end{equation}
one obtains a   "canonical" boundary {triplet} $\{\cH, \Gamma_0^0, \Gamma_1^0\}$
constructed  in   \cite{Koc75} where  $P_{\pm i}$ are projections from $\dom(A^*)$ onto
$\mathfrak{N}_{\pm i}$ parallel to $\dom(A)+\mathfrak{N}_{\mp i}$. Clearly,
$\widetilde{A}_0$ and $\widetilde{A}_1$  are given by formulas~\eqref{DM_3}.

In what follows, by abuse of language,  a triplet $\{\cH, \Gamma_0^0, \Gamma_1^0\}$ will
be called  a "canonical" boundary {triplet} corresponding to the operator
$\widetilde{A}_0=\widetilde{A}_0^*$ (or $\widetilde{A}_1$).
   \begin{proposition}\label{P:DM_1}
Let $\{\cH, \Gamma_0, \Gamma_1\}$ and $\{\cH, \Gamma_0, \Gamma_1'\}$ be two boundary
{triplet}s with a common boundary space $\cH$ and the operator $\Gamma_0$. Then there
exists an  operator  $K=K^*\in[{\mathcal H}]$, such that
  \begin{equation}\label{E:10.18}
   \Gamma_1=\Gamma_1'+K\Gamma_0.
  \end{equation}
\end{proposition}
     \begin{proof}
Taking a difference of two  equalities \eqref{DM_1} written down for each of  boundary
{triplet}s  yields
\begin{equation}\label{DM_7}
(\Gamma_0f, (\Gamma_1-\Gamma_1')g)=((\Gamma_1-\Gamma_1')f, \Gamma_0g), \qquad
f,g\in\dom(A^*).
 \end{equation}
Let us establish the implication:
\begin{equation}\label{eq:Impl}
    \Gamma_0g=0 \quad \Rightarrow \quad \Gamma_1g=\Gamma_1'g
\end{equation}
Let us choose a vector $f\in\dom(A^*)$, such that $\Gamma_0f=(\Gamma_1-\Gamma_1')g$.
Inserting such vector  $f$ into~\eqref{DM_7}  implies
$$
\|(\Gamma_1-\Gamma_1')g\|^2=((\Gamma_1-\Gamma_1')g,\Gamma_0g)=0 \quad \Longrightarrow
\quad \Gamma_1g=\Gamma_1'g.
$$
Let us define the operator $K:\cH\rightarrow\cH$ by  setting
\[
K\Gamma_0f=\Gamma_1f-\Gamma_1'f\quad (f\in\dom(A^*)).
\]
In view of the implication~\eqref{eq:Impl} the operator $K$ is well defined. It follows from~\eqref{DM_7} that the operator $K$ is symmetric
\[
(\Gamma_0f, K\Gamma_0g)_\cH=(K\Gamma_0f, \Gamma_0g)_\cH,
\]
and since $\dom(K)=\cH$, it is bounded and selfadjoint,  $K=K^*\in[\cH]$.
\end{proof}
\begin{corollary}\label{C:DM_1}
Let under the assumptions of Proposition~\ref{P:DM_1} the operators ${A}_i={A}_i^*
\;(i=0,1)$ are defined by the equality~\eqref{DM_3} and $\dom(A_1')=\ker\Gamma_1'$. Then
the operators $\tilde{A}_1$ and $\tilde{A}_1'$ are transversal if and only if the
operator $K$ is boundedly invertible, i.e. $K^{-1}\in[\cH]$. In this case the triplet
$\{\cH,\Gamma_0,\Gamma_1'\}$, with $\Gamma_1'=K^{-1}\Gamma_1$ is also a boundary
{triplet}.
\end{corollary}
\begin{proof}
The invertibility of the operator $K$ in a wide sense is equivalent to the relation:
$\dom({A}_1)\cap\dom({A}_1')=\dom({A})$. The rest is evident.
\end{proof}

Let $\{\cH,\Gamma_0,\Gamma_1\}$ be a boundary {triplet} for the operator $A^*$. Endowing
$\dom(A^*)$ with the graph norm one obtains a Hilbert space, in which every subspace
\footnote[1]{Subspace means a closed linear subspace} $\mathfrak{M}$, containing
$\dom(A)$, determines  a closed proper extension of the operator $A$. The mapping
$\Gamma:y\rightarrow
\begin{pmatrix}
    \Gamma_1y\\ \Gamma_0y\end{pmatrix}$ determines a topological isomorphism between
$\dom(A^*)/\dom(A)$ and $\cH^2:=\begin{pmatrix}
    \cH\\ \cH\end{pmatrix}$
    which establishes a bijective correspondence between closed proper extensions $\wt A$ of the operator
$A$ and subspaces in $\cH^2$.
\begin{equation}\label{DM_8}
A\subset\wt{A} \quad \longleftrightarrow \quad
\mathfrak{M}_{\wt{A}}=\{\Gamma y: y\in\dom(\wt{A})\}
\end{equation}
\begin{proposition}\label{P:DM_2}
A proper extension  $\wt{A}$ of the operator $A$ is almost solvable  if and only if it
is transversal to some self-adjoint extension $\wt{A}_0=\wt{A}_0^*$  of the operator
$A$. In this case for every boundary {triplet} $\{\cH,\Gamma_0,\Gamma_1\}$, for which
$\ker\Gamma_0=\dom(\wt{A}_0)$, there exists an operator $B\in[\cH]$, such that
$\wt{A}={A}_B$, i.e.~\eqref{DM_2} holds.
\end{proposition}
\begin{proof}
Let $\wt{A}$ be an a.s. extension of the operator $A$. Then there exists a boundary
{triplet} $\{\cH,\Gamma_0,\Gamma_1\}$, such that
$\dom(\wt{A})=\ker(\Gamma_1-B\Gamma_0)$, i.e. $\wt{A}={A}_B$. Define
$\wt{A}_0=\wt{A}_0^*$, by setting
\[
f\in\dom(\wt{A}_0)\;\Leftrightarrow\;
f\in\dom(A^*)\quad\mbox{и}\quad\Gamma_0f=0,
\]
and check the transversality of the extensions
${A}_B$ and $\wt{A}_0$. Clearly, the condition
 \[
 \dom(\wt{A}_0)\cap\dom({A}_B)=\dom(A),
 \]
is fulfilled. Let us check the condition
\begin{equation}\label{DM_5A}
     \dom(\wt{A}_0)+\dom({A}_B)=\dom(A^*),
\end{equation}
assuming that $f\in\dom(A^*)$ and $\Gamma_0f=\varphi\neq0$
 (since $\Gamma_0f=0$ implies $f\in\dom(\wt{A}_0)$).
Let us choose a vector $g\in\dom(A^*)$, such that $\Gamma_0f=\varphi$ and
$\Gamma_1f=B\varphi$. Then $g\in\dom(\wt A_B)$, since
$\Gamma_1f-B\Gamma_0 g=B\varphi-B\varphi=0$, and  $f-g\in\dom(\wt
A_0)$, since $\Gamma_0(f-g)=\varphi-\varphi=0$. Therefore,
\[
f=g+(f-g)\in\dom(\wt A_B)+\dom(\wt A_0)
\]
and hence~\eqref{DM_5A} is proven.

Let us prove the converse statement. Let $\wt{A}_1=\wt{A}_1^*$ be an operator canonically transversal to the operator $\wt{A}_0$, let
$\{\cH,\Gamma_0^0,\Gamma_1^0\}$ be the corresponding boundary {triplet} of the form~\eqref{DM_6}, and let $\mathfrak{M}=\wt{\mathfrak{M}}_{\wt{A}}$
be a subspace of $\cH^0_0\oplus\cH^0_1$, corresponding to $\wt{A}$ in the sense of~\eqref{DM_8}. It is easy to see that the transversality of the operators
$\widetilde{A}_0$ and $\widetilde{A}_1$ is equivalent to the equality
\begin{equation}\label{DM_9}
\cH_1^0\dotplus\mathfrak{M}=\cH^0_0\oplus\cH^0_1.
\end{equation}
In turn,  \eqref{DM_9}  holds if and only if $\mathfrak{M}$ is a graph of a bounded
operator $B:\cH^0_0\rightarrow\cH^0_1$. Indeed, if  $\mathfrak{M}$ is a graph of a
bounded operator, then clearly \eqref{DM_9} is in force. Conversely, if the sum
in~\eqref{DM_9} is direct, then $\mathfrak{M}$ has no elements of the form
 $(x,0)$, i.e. $\mathfrak{M}$ is a graph of a closed operator  $B:\cH^0_0\rightarrow\cH^0_1$. Further, if $\dom(B)\neq\cH^0_0$, then there exists
$x_2\in\cH^0_0, \; x_2\in\dom(B)$, and hence $(0,x_2)\in\cH^0_1\dotplus\mathfrak{M}$, what contradicts to~\eqref{DM_9}. This proves the statement.
\end{proof}
\begin{corollary}\label{C:DM_2}
Let $\wt{A}_1=\wt{A}_1^*$, and let $\wt{A}_0=\wt{A}_0^*$ be transversal extensions of the operator $A$. Then there exists a boundary {triplet}
$\{\cH,\Gamma_0,\Gamma_1\}$, such that
\eqref{DM_3} are in force.
\end{corollary}
\begin{proposition}\label{P:DM_3}
Let $\wt{A}_0$ and $\wt{A}_1$ be proper extensions of the operator
$A$, having common regular points $z_1\in\dC_+$ and $z_2\in\dC_-$ and let the operators $X_{z}:\mathfrak{N}_{\bar{z}}\rightarrow\mathfrak{N}_{z}$ be defined by
\begin{equation}\label{DM_10}
X_z=\left[(\wt{A}_1-z)^{-1}-(\wt{A}_0-z)^{-1}\right]_{\mathfrak{N}_{\bar{z}}}
\quad (z\in\dC_+\cup\dC_-).
\end{equation}
Then the  operators $\wt{A}_0$ and $\wt{A}_1$ are transversal if and and only if
the operators $X_{z_1}$ and $X_{z_2}$
are boundedly invertible, i.e.
\begin{equation}\label{eq:Trans}
    \ker X_{z_k}=0\quad\mbox{ and }\quad
    X_{z_k}^{-1}\in[\mathfrak{N}_{z_k},\mathfrak{N}_{\bar{z}_k}]
    \quad
(k=1,2).
\end{equation}
\end{proposition}
\begin{proof}
Notice first, that the condition $\ker X_z=0$ $(z\in\rho(A_0)\cap\rho(A_1))$
is equivalent to the condition \eqref{DM_4}.

If the extensions $\wt{A}_0$ and $\wt{A}_1$ are transversal, then in view of the condition
\eqref{DM_5} for every vector $f\in H$ there are vectors
$\varphi_0$ and $\varphi_1\in H$, such that
$(\wt{A}_0-z)^{-1}\varphi_0+(\wt{A}_1-z)^{-1}\varphi_1=f$. If
$f\in\mathfrak{N}_z$, then applying the operator
$A^*-z$ to the latter equality, one obtains $\varphi_0+\varphi_1=0 \; \Rightarrow \;
\varphi_1=-\varphi_0$, i.e.
\[
(\wt{A}_1-z)^{-1}\varphi_1-(\wt{A}_0-z)^{-1}\varphi_1=f.
\]
In this relation one can choose
$\varphi_1\in\mathfrak{N}_{\bar{z}}$, as follows from the equality
$H=\mathfrak{M}_z\oplus\mathfrak{N}_{\bar{z}}$.
Hence for $z\in\rho(A_0)\cap\rho(A_1))$
one has $X_z\mathfrak{N}_{\bar{z}}=\mathfrak{N}_z$
and, therefore, the operator  $X_z$ is invertible.

The converse statement follows from the generalized J.von Neumann formula
\begin{equation}\label{DM_11}
\dom(A^*)=\dom(A)\dotplus\mathfrak{N}_{z_1}\dotplus\mathfrak{N}_{z_2}
\quad(\IM z_1\cdot\IM z_2<0),
\end{equation}
which is derived from the relation $H=\sM_{z_1}\dotplus\sN_{\bar z_2}$
in the same manner as J.von Neumann formula is reduced from the equality
$H=\sM_{z}\dotplus\sN_{\bar
z}$. If now the conditions~\eqref{eq:Trans} hold, then
\[
\mathfrak{N}_{{z}_k}=X_{z_k}\mathfrak{N}_{\bar{z}_k}\subset\dom(\wt
A_0)+\dom(\wt A_1) \quad(k=1,2)
\]
and by the formula~\eqref{DM_11} the extensions $\wt A_0$ and $\wt A_1$
are transversal.
\end{proof}
\begin{corollary}\label{cor:DM_3}
Let $\wt{A}_0$ and $\wt{A}_1$ be proper extensions of the operator
$A$, having common regular real  point  $a\in\dR$. Then the transversality of the operators $\wt{A}_0$ and $\wt{A}_1$ is equivalent to the following conditions
\begin{equation}\label{eq:Trans_2}
    \ker X_{a}=0\quad\mbox{ and }\quad
    X_{a}^{-1}\in[\mathfrak{N}_{a}].
\end{equation}
\end{corollary}
The proof is implied by the fact, that the operators  $X_z$
are well defined and invertible for all  $z$ which are close enough to $a$.


\begin{proposition}\label{P:DM_5}
Let a proper extension $\widetilde{A}$ of the operator $ A$ have two regular points  $z_1$, $z_2\in\rho(\widetilde{A})$, such that
$\Im{z_1}\cdot\Im{z_2}<0$. Then $\widetilde{A}$ is an almost solvable extension of the operator $ A$.
\end{proposition}
\begin{proof}
Assume for definitness that $z_1=-i$, $z_2\in\dC_+$. Then (see
\cite{ShmTs77})
\begin{equation}\label{eq:Mext}
    \dom(\wt{A})=\dom{A}+(I+M)\sN_i, \quad\mbox{ where }
M=(\wt{A}-i)(\wt{A}+i)^{-1}\upharpoonright_{\sN_i}\in[\sN_i,\sN_{-i}].
\end{equation}
According to A. V. Strauss  \cite{Str68}, the characteristic function $C(z)$ of the operator $A$ $(C(z)\in[\sN_i,\sN_{-i}])$ is defined by the relation:
\begin{equation}\label{DM_14}
\dom(A)+\sN_z=\dom(A)+(I+C(z))\sN_i.
\end{equation}
The conditions $z_2\in\rho(\wt{A})$ and $0\in\rho(C(z_2)-M)$ are equivalent
(see \cite{Koc80} and Section \ref{4}). Therefore, the operator $C(z_2)-M$ has the polar representation $C(z_2)-M=VR$, where $R>0$, and $V$ is an isometry from $\sN_i$ onto $\sN_{-i}$.

Since $\|C(z)\|< 1$ for all $z\in{\dC}_+$ (see \cite{Str68}), then
$\mbox{Re }(V^*C(z_2))< I$ and hence the operator
$I-V^*M=I-V^*C(z_2)+R$ is boundedly invertible, i.e. $1\in\rho(-U^*M)$.
Define a boundary {triplet} by~\eqref{DM_6}. Then for a vector $f\in\dom(\wt A)$ of the form
\[
f=f_0+(I+M)f_i\quad(f_0\in\dom(A),\,f_i\in\sN_i)
\]
one obtains
\[
\Gamma_0^0f=(V-M)f_i,\quad\Gamma_1^0f=i(V+M)f_i.
\]
Hence, one arrives at the equality
$\dom(\wt{A})=\ker(\Gamma_1^0-B\Gamma_0^0)$, in which
$B=i(M+V)(V-M)^{-1}\in[\cH]$. Therefore, the extension
$\widetilde{A}$ of the operator $A$ is almost solvable.
\end{proof}

\begin{remark}\label{R:DM_1}
It follows from the above proof and Proposition~\ref{P:DM_7} that the extension
 $\wt{A}$ is almost solvable, if there exist
$z_1,z_2\in\rho(\wt{A})\cup\sigma_c(\wt{A})$, such that $\Im
z_1\cdot\Im z_2<0$.
\end{remark}

\begin{remark}\label{R:DM_2}
Finally, let us present examples of proper extensions
$\wt{A}$ of the operator $A$, which are not almost solvable.
To show this on account of Proposition~\ref{P:DM_2}, it is enough to point out a proper extension
$\wt{A}$ of the operator $A$, which is not transversal to any self-adjoint extension  $\wt A_0=\wt{A}_0^*$ . Let the domains
of the extensions $\wt{A}$  and $\wt A_0$ are given by
\begin{equation}\label{eq:wtA}
    \dom(\wt{A})=\dom(A)+(I+M)\sN_i,\quad M\in[\sN_i,\sN_{-i}],
\end{equation}
\begin{equation}\label{eq:wtA_0}
    \dom(\wt{A}_0)=\dom(A)+(I+V)\sN_i,\quad V\in[\sN_i,\sN_{-i}].
\end{equation}
Clearly, the disjointness of the extensions $\wt{A}$
and $\wt A_0$ is equivalent to the condition:
$1\not\in\sigma_p(V^*M)$, and the transversality of the extensions  $\wt{A}$
and $\wt A_0$ is equivalent to the condition:
$1\in\rho(V^*M)$. Notice also, that  $-i\in\rho(\wt A)$.
\end{remark}

one can construct the needed examples by setting
 $M=\alpha U^*$, where $|\alpha|>1$ and  $U$  is an isometry from
$\dom(U)=\sN_i$ onto $U\sN_{-i}\nsubseteqq\sN_{i}$. Indeed, for any isometry
$V$ from $\sN_i$ onto $\sN_{-i}$ the operator $UV$ is a nonunitary isometry
in $\sN_i$, and by the Wald decomposition
$UV = U_0\oplus U_1$, where $U_0$ is a unitary operator, and
  $U_1 (\not =
0)$ is a unilateral shift. Therefore, the point spectrum of the operator
$V^*U^* = (UV)^*$
coincides with the open unit disc. Indeed,
$1\in\sigma_p(V^*M)=\sigma_p(\alpha V^*U^*)$ for $|\alpha|>1$, and, therefore, the extension
$\wt A$, determined by the relation~\eqref{eq:wtA}, is not disjoint to to any self-adjoint extension
of the operator $A$. Thus, the extension $\wt A$ cannot be represented in the form~\eqref{DM_2}, where $B\in\cC(\cH)$ and, in particular, is not almost solvable.
At the same time the extension $\wt A$ is transversal to the extension
 ${\wt A}_{-i}$, for which  $\dom({\wt A}_{-i}) = \dom(A) + \sN_{-i}$ and, hence, $-i\in\rho(\wt A)$.

This example shows that the assumptions of Proposition~\ref{P:DM_5} are essential. It is also worth mentioning
that the operators $\wt{A}$ and $\wt{A}^*$ are transversal for $|\alpha|>1$. Indeed, the conditions
\[
\dom(\wt{A})\cap\dom(\wt{A}^*)=\dom(A),\quad
\dom(\wt{A})+\dom(\wt{A}^*)=\dom(A^*)
\]
are equivalent to the following ones:
\[
1\in\rho(M^*M),\quad 1\in\rho(MM^*).
\]
The latter conditions certainly are fulfilled, since $M^*M=|\alpha|^2I$ and
$MM^*=|\alpha|^2P$ ($P$ is an orthoprojection). Therefore, the transversality of the operators $\wt{A}$ and $\wt{A}^*$ is not sufficient for the extension  $\wt{A}$ to be almost solvable.

If $|\alpha|=1$ then the operator  $\wt A$, presented above, is a maximal symmetric (but not self-adjoint $\wt{A}\neq\wt{A}^*$) extension of the operator $A$, which is not almost solvable, since
$1\in{\sigma}_c{(V^*M)}=\sigma_c(\alpha V^*U^*)$ при $|\alpha|=1$.

\section{Weyl Functions and $Q$--Functions of Hermitian Operators}\label{2}

\textbf{1.} Let $\Pi=\{\cH, \Gamma_0,\Gamma_1\}$ be some boundary triplet
for the operator $A^*$, and let ${A}_0={A}_0^*$, ${A}_1={A}_1^*$
be extensions corresponding to the operators $\Gamma_0$ and
$\Gamma_1$ in the sense of relations~\eqref{DM_3}.
\begin{definition}\label{D:DM_4}
An operator function $M(z)$ defined by
\begin{equation}\label{DM_15}
M(z)\Gamma_0f_z=\Gamma_1f_z\qquad (f_z\in\sN_z,\ z\in{\rho}(A_0))
\end{equation}
is said to be a  Weyl function of the operator $A$ corresponding to the boundary triplet $\Pi$.
\end{definition}

Let us show that the operator function $M(z)$ is well defined, confining ourselves to the case when $\Im z>0$. To this end, we notice that for $\Im z>0$ the proper extension
$\wt{A}_z\subset A^*$ of the operator $A$ defined by
\[
\dom(\wt{A}_z)=\dom(A)+\sN_z, \qquad z\in \dC_+,
\]
is a closed, maximal dissipative\footnote[1]{This fact is proved in~\cite{Str68}.
Here we suggest its elementary proof. Since $\wt{A}_z(f_A+f_z)=Af_A + zf_z$,
where $f_A \in\dom(A),  f_z\in\sN_z$, it follows that
$(\wt{A}_zf,g)_H-(f,\wt{A}_zg)_H=(z-\bar{z})(f_z,g_z)_H$ and, therefore, $\wt{A}_z$ is dissipative. Clearly, $\wt{A}_z$ is closed. Moreover,
the fact that $\wt{A}_z$ is maximal dissipative follows from the relation
$(\wt{A} -z_0)\dom(\wt{A})=H,\  \Im z_0<0$. The latter can be set, for example, as follows.
Assuming the contrary, one obtains the equality $((\wt{A} -z_0)f, g)_H =0$ for every $f \in\dom(\wt{A})$ with some $g\in H$. On the one hand, for $f = f_A \in\dom(A)$
this implies  $g\in \sN_{\bar{z_0}}$ and, on the other hand, for $f \in\sN_z$ the same equality yields $g\perp\sN_{z}$.
However, for $\bar{z}_0$ close to $z$ the relation $g\perp\sN_z$
contradicts the fact that the aperture of the subspaces $\sN_z$ and $\sN_{\bar{z_0}}$ is less than one~\cite[Section~34]{AG81}.}
extension of $A$, which
and, in view of the Neumann formulas, is transversal to any self-adjoint extension $\wt{A}$,
in particular, to the operator ${A}_0$. By Proposition~\ref{P:DM_2}, we have
$\wt{A}_z={A}_{B(z)}$, where $B(z)$ belongs to $[\cH]$ and is dissipative. Hence,
\[
f\in\dom(\wt{A}_z)\Longleftrightarrow f\in\dom(A^*)\quad\mbox{ and }\quad
B(z)\Gamma_0f=\Gamma_1f.
\]
But $\Gamma_0(\dom(A^*))=\cH$ by the definition of a boundary triplet. Taking into account
the transversality of the operators ${A}_0$ and  $\wt{A}_z$ this implies the relation:
\[
\Gamma_0\sN_z=\Gamma_0(\dom(\wt{A}_z))=\Gamma_0(\dom(A^*))=\cH,
\]
which leads to the equality $M(z)=B(z)$. Thus, $M(z)$ is the operator function
with values in $[\cH]$ and $\Im z\cdot\Im M(z)>0$ whenever $z\in\wh{\rho}(A)$.
Together with the analyticity of $M(z)$ in $\dC_+$  proved below this ensures
that $M(z)$ belongs to the class $(R)$. Note also that $M(z)^*=M(\bar{z})$ since
$(\wt A_z)^*=\wt{A}_{\bar{z}}$.

\begin{remark}\label{R:DM_3}
When justifying that the Weyl function is well defined it has been shown
that the operators $\Gamma_0$ and $\Gamma_1$ map $\sN_z$ onto $\cH$ isomorphically.
This fact is also extracted from the results of~\cite{Br76} and is useful for specific operators in Section~\ref{2}. Let us set
\begin{equation}\label{eq:gamma}
    \gamma(z):=(\Gamma_0|_{\sN_z})^{-1}\quad (z\in\rho(A_0)).
\end{equation}
The operator function $\gamma(z)$ takes values in $[\cH,\sN_z]$
for any $z\in\rho(A_0)$.
\end{remark}

Clearly, the function $M(z)$ depends on the choice of a boundary triplet.
In view of Proposition~\ref{P:DM_1}, we obtain the following connection between
functions $M(z)$ and $\wt M(z)$ corresponding to two boundary
triplets $\{\cH,\Gamma_0,\Gamma_1\}$ and $\{\cH, \wt\Gamma_0,\Gamma_1\}$
of the operator $A^*$ with a common operator $\Gamma_1$:
\begin{equation}\label{DM_16}
\wt M^{-1}(z)=M^{-1}(z)+K,\qquad K=K^*\in[\cH].
\end{equation}

To clarify this connection in the general case, denote by $M(z)$ and
$\wt M(z)$ the Weyl functions of the operator $A$ corresponding to the boundary triplets
$\{\cH, \Gamma_0,\Gamma_1\}$ and $\{\wt\cH,\wt\Gamma_0,\wt\Gamma_1\}$
respectively. Let also $U$ be an isometric operator from $\cH$ onto $\wt\cH$
$(\dim(\cH)=\dim(\wt\cH))$, and let
\[
J=\left(
 \begin{array}{cc}
 0 & iI_\cH \\
-iI_\cH & 0 \\
\end{array}
\right)
\quad \mbox{ be the signature operator in }\quad\cH\oplus\cH.
\]
\begin{proposition}\label{P:DM_6}
Suppose that $\Pi=\{{\mathcal H},\Gamma_0,\Gamma_1\}$ and
$\wt\Pi=\{\wt{{\mathcal H}},\wt{\Gamma}_0,\wt{\Gamma}_1\}$ are some boundary triplets
of the operator $A^*$, and $U$ is an isometric operator from
${\mathcal H}$ onto $\wt{{\mathcal H}}$. Then the operators $\Gamma$ and
$\wt\Gamma$ are related by
\begin{equation}\label{DM_18}
\begin{pmatrix}\wt\Gamma_1\\
\wt\Gamma_0\end{pmatrix}= \left(
               \begin{array}{cc}
                 U & 0 \\
                 0 & U \\
               \end{array}
             \right)
\left(
               \begin{array}{cc}
                 X_{11} & X_{12} \\
                 X_{21} & X_{22} \\
               \end{array}
             \right)\begin{pmatrix}\Gamma_1\\
\Gamma_0\end{pmatrix},
\end{equation}
where $X=(X_{ij})_{i,j=1}^2$,\   is a $J$-unitary operator in $\cH\oplus\cH$, and
the Weyl functions $M(z)$ and $\wt{M}(z)$ corresponding to the boundary triplets
$\Pi$ and $\wt\Pi$ satisfy the relation
\begin{equation}\label{DM_17}
\wt M(z)=U(X_{11}M(z)+X_{12})(X_{21}M(z)+X_{22})^{-1}U^{*}.
\end{equation}
\end{proposition}

\begin{proof}
The proof is similar to the proof of Theorem~\ref{T:DM_1} from~\cite{Koc80}.
Define the operator
$X:\cH\oplus\cH\rightarrow\cH\oplus\cH$ by relation~\eqref{DM_18}.
It is clear that the operator $X$ is well defined and surjective. The definition
of a boundary triplet implies that the operator $X$ is $J$--unitary.
It follows that $X$ is bounded (see~\cite{Iokhv}), and
relation~\eqref{DM_18} takes the form
\begin{equation}\label{DM_18A}
\wt\Gamma_1f=U(X_{11}\Gamma_1f+X_{12}\Gamma_0f),\quad
\wt\Gamma_0f=U(X_{21}\Gamma_1f+X_{22}\Gamma_0f)
\end{equation}
with operators $X_{ij}\in[\cH]$. By definition of the Weyl function,
\begin{equation}\label{DM_19}
\wt M(z)\wt\Gamma_0f=U\left(X_{11}M(z)+X_{12}\right)\Gamma_0f,\qquad
\wt\Gamma_0f=U\left(X_{21}M(z)+X_{22}\right)\Gamma_0f.
\end{equation}
Since the operators $M(z)$ are dissipative for $\Im z>0$,
the $J$--unitarity of the operator $X$ and the Krein--Shmuljan theorem~\cite{KrSh74}
together imply that the operator $X_{22}M(z)+X_{22}$ has a bounded inverse. Hence,
in view of the definition~\eqref{DM_18}, $\wt M(z)$ is of the form~\eqref{DM_17}.
The proof for $\Im z<0$ is similar.
\end{proof}

\begin{corollary}\label{C:DM_4}
The $J_{\wt{\mathcal H}}$--unitarity of the operator $X$ implies the relations
\begin{equation}\label{eq:1.34}
   \begin{array}{ccc}
     X_{11}^*X_{21}=X_{21}^*X_{11},\qquad & X_{12}^*X_{22}=X_{22}^*X_{12},\qquad &
     X_{11}^*X_{22}-X_{21}^*X_{12}=I_{\mathcal H},
   \end{array}
\end{equation}
\begin{equation}\label{eq:1.35}
   \begin{array}{ccc}
     X_{11}X_{12}^*=X_{12}X_{11}^*,\qquad & X_{21}X_{22}^*=X_{22}X_{21}^*,\qquad &
     X_{11}X_{22}^*-X_{12}X_{21}^*=I_{\mathcal H}.
   \end{array}
\end{equation}

\end{corollary}

\begin{corollary}\label{C:DM_4A}
The function $M(z)$ belongs to the class $(R)$.
\end{corollary}

\begin{proof}
First, we will prove this statement  for the function $M_0(z)$ coresponding to
a "canonical"\ boundary triplet of the form~\eqref{DM_6}. Representing $f_z\in\sN_z$
as $f_z=({A}_0+i)({A}_0-z)^{-1}f_{-i}$ we obtain from~\eqref{DM_6} that
\[
\Gamma_0^0f_z=-f_{-i},\qquad
\Gamma_1^0f_z=if_{-i}+(z+i)\Gamma_1^0({A}_0-z)^{-1}f_{-i}.
\]
Therefore, the corresponding Weyl function is of the form
\begin{equation}\label{DM_20}
M_0(z)=-iI-(z+i)\Gamma_1^0({A}_0-z)^{-1}.
\end{equation}
Relation~\eqref{DM_20} implies the analyticity of $M_0(z)$
in the domain $\rho({A}_0)$, in particular, for $\bar{z}\neq z$, and, by
 The analyticity of $M(z)$ for an arbitrary
boundary triplet follows from Proposition~\ref{P:DM_6}. The property $\Im M(z)\cdot\Im z>0$ for $\bar{z}\neq z$
has been established before.
\end{proof}
\begin{example}\label{ex:3.6}
Let $A$ be a minimal symmetric Sturm--Liouville operator
determined in $L^2(0,\infty)$ by the differential expression $Ay=-y''+q(x)y$
with a bounded potential $q(x)(=\overline{q(x)})$. Then the operator $A$ is in the limit point case at $\infty$, and $0$ is a regular endpoint for the operator $A$ (see~\cite{Naj69}). In this case,
 \[\dom(A^*)=W^{2,2}[0,\infty),\quad
\dom(A)=\{y\in W^{2,2}[0,\infty):\,y(0)=y'(0)=0\}.
 \]
 Introduce the one-parametric family of boundary triplets
$\{\Gamma_{0h},\Gamma_{1h},\dC\}$ $(h\in\dR\cup{\infty})$ by setting
\[
\Gamma_{0h}y =(y'(0)-hy(0))(1+h^2)^{-1/2},\qquad
\Gamma_{1h}y=-(hy'(0)+y(0))(1+h^2)^{-1/2}
\]
for $h\in\dR$ and
\[
\Gamma_{0,\infty}y =y(0),\qquad \Gamma_{1,\infty}y=y'(0)
\]
for $h=\infty$.

Straightforward calculations show that $M_h(z)$ coincides with the classical Weyl
function\footnote[1]{$M_h(z)$ differs in sign from $m_h(z)$ used in \cite{Marc52,Pavlov66}.} $m_h(z)$ (see~\cite{Naj69}), and formula~\eqref{DM_17} acquiring the form
\[
m_h(z)=(1-hm_\infty(z))(m_\infty(z)-h)^{-1},
\]
expresses the well-known relationship between two Weyl functions  (see~\cite{Marc52},\cite{Pavlov66}).
\end{example}

\textbf{2.} Now recall the definition of the $Q$--function of
a Hermitian operator~\cite{KL71}. Assume that $\wt{A}=\wt{A}^*$ is a self-adjoint extension of
the operator $A$ and $\cH$ is an auxiliary Hilbert space
$(\dim\cH=n\leqslant\infty)$. Assume also that $\gamma(z_0)$ is an operator from  $[\cH,\sN_{z_0}]$ such that $\gamma(z_0)^{-1}\in[\sN_{z_0},\cH]$. The relation
\begin{equation}\label{eq:gamma_f}
    \gamma(z)=(\wt{A}-z_0)(\wt{A}-z)^{-1}\gamma(z_0),\qquad
z,z_0\in\rho(\wt{A}),
\end{equation}
determines the analytic vector function with values in $[\cH,\sN_{z}]$
which is called  a $\gamma$-field of the operator $A$ (see~\cite{KL71}).
It is easy to see that the operator function $\gamma(z)$ defined by the relation~\eqref{eq:gamma}
satisfies the identity~\eqref{eq:gamma_f} and hence it is the $\gamma$-field of
the operator $A$.
\begin{definition}[{[30]}]\label{D:DM_5}
An operator function $Q(z)$ with values in $[\cH]$ is called the $Q$--function of
the operator $A$ belonging to the proper extension $\wt{A}$ if for any $z,\zeta\in\rho(\wt{A})$
the following equality holds
\begin{equation}\label{DM_21}
Q(z)-Q(\zeta)^*=(z-\bar\zeta)\gamma(\zeta)^*\gamma(z).
\end{equation}
\end{definition}
\begin{theorem}\label{T:DM_1}
The Weyl function $M(z)$ corresponding to the boundary triplet
$\{\cH,\Gamma_0,\Gamma_1\}$ is the $Q$--function of the operator $A$
belonging to the self-adjoint extension\footnote[2]{Recall that
$f\in\dom({A}_0)\Leftrightarrow f\in\dom(A^*)$ and $\Gamma_0f=0$.}
$A_0$.
\end{theorem}

\begin{proof}
First we give the proof for the "canonical"\ boundary triplet of form~\eqref{DM_6}.
In this case, the Weyl function has form~\eqref{DM_20}. It is clear that
\begin{equation}\label{DM_22}
\Gamma_1^0=-P({A}_0-i)P',
\end{equation}
where $P$ is the orthoprojector onto $\sN_{-i}$, and $P'$ is the projector from
$\dom(A^*)$ onto $\dom({A}_0)=\dom(A)\dot{+}(I+V)\sN_i$ in the decomposition
\[
\dom(A^*)=\dom({A}_0)\dot{+}(I-V)\sN_{i}.
\]
Hence, both formulas~\eqref{DM_19} and~\eqref{DM_20} imply that
\begin{eqnarray*}
M_0(z)&=&-iI+P(z+i)({A}_0-i)({A}_0-z)^{-1}\\
&=& P[zI+(z^2+1)({A}_0-z)^{-1}]=\\
&=& P[iI+(z-i)({A}_0+i)({A}_0-z)^{-1}].
\end{eqnarray*}
By setting $\gamma(-i)=I|_{\sN_{-i}}$ and by
taking into account that $\gamma(-i)^*=P$ one obtains the relation
\begin{eqnarray*}
M_0(z)&=&i\gamma(-i)^*\gamma(-i)+(z-i)\gamma(-i)^*\gamma(z)\\
&=&M_0(-i)^*+(z-i)\gamma(-i)^*\gamma(z),
\end{eqnarray*}
which yields the relation~\eqref{DM_21}.

Now let $\{\cH, \Gamma_0,\Gamma_1\}$ be an arbitrary boundary triplet.
Together with it, consider the "canonical"\ boundary triplet
$\{\cH^0, \Gamma_0^0,\Gamma_1^0\}$ of the form~\eqref{DM_6} and the corresponding
proper extension ${A}_0$. Let $U$ be a unitary operator from $\cH$ to $\sN_{-i}$.
The formulas that connect $\{\cH, \Gamma_0,\Gamma_1\}$ and $\{\cH^0, \Gamma_0^0,\Gamma_1^0\}$
take the form
\[
\Gamma_1=U(X_{11}\Gamma_1^0+X_{12}\Gamma_0^0), \quad
\Gamma_0=UX_{22}\Gamma_0^0,
\]
since $X_{21}=0$ in view of the condition $\ker(\Gamma_0)=\ker(\Gamma_0^{(0)})$.
Since the operator $X$ is $J$--unitary then  $X_{22}^{-1}=X_{11}^*$,
$X_{12}X_{11}^*=(X_{12}X_{11}^*)^*$.  By virtue of these relations and the formula~\eqref{DM_17}
one gets the equality
\begin{eqnarray*}
M(z)&=&U[X_{11}M_0(z)+X_{12}]X_{22}^{-1}U^*\\
&=& UX_{11}M_0(z)X_{11}^*U+U^{-1}X_{12}X_{11}^*U^*.
\end{eqnarray*}
Hence, we have
\begin{equation}\label{DM_23}
M(z)=CM_0(z)C^*+D,
\end{equation}
where $C=U^{-1}X_{11}$, $D=U^{-1}X_{12}X_{11}^*U$, and the function $M(z)$
is also the $Q$--function of the operator $A$ belonging to the extension ${A}_0$.
\end{proof}
\begin{remark}\label{R:DM_4}
In what follows the function $M(z)$ will be also called the Weyl
function of the operator ${A}_0$. Thus, two Weyl functions of the operator ${A}_0$
are related by~\eqref{DM_23}.
\end{remark}

Properties of the $Q$--function obtained in~\cite{KL73} make it possible to formulate the
following corollaries.
\begin{corollary}\label{C:DM_5}
Simple Hermitian operators $A{'}$ and $A{''}$ are unitary equivalent if and only if,
for some choice of boundary triplets $\{\cH^1,\Gamma_0',\Gamma_1'\}$ and
$\{\cH^2, \Gamma_0'',\Gamma_1''\}$ for $(A')^*$ and $(A'')^*$, respectively,
their Weyl functions coincide. In this case, the extensions\footnote[1]{$\dom(A_0')=\ker\Gamma_0',\
\dom(A_0'')=\ker\Gamma_0''$.} ${A}_0'$ and ${A}_0''$ are also unitary equivalent.
\end{corollary}

\begin{corollary}\label{C:DM_6}
For an operator function $M(z)$ with values in $[\cH]$ analytic on the upper half-plane 
 to be the Weyl function of a simple densely defined
Hermitian operator $A$, it is necessary and sufficient that the following
three conditions hold:
\begin{itemize}
\item[(i)] $M\in(R)$;
\item[(ii)] $w-\lim\limits_{y\uparrow\infty}\frac{M(iy)}{y}=0$;
\item[(iii)] $\lim\limits_{y\uparrow\infty}y\Im(M(iy)h,h)=\infty$  for any  $h\in\cH\setminus\{0\}$.
\end{itemize}
\end{corollary}
\section{Spectra of Extensions and Weyl Function}\label{3}

In this section we describe the spectrum of the operator ${A}_B$ in
terms of the Weyl function and establish a criterion of a resolvent comparability
of two extensions ${A}_{B_1}$ and ${A}_{B_2}$.
\begin{proposition}\label{P:DM_7}
Let $\{\cH, \Gamma_0,\Gamma_1\}$ be a boundary triplet, and let
$B\in\cC(\cH)$, ${A}_B\supset A$, $z\in\hat{\rho}(A)$. Then:
\begin{itemize}
\item[(i)] $z\in\sigma_p({A}_B)\Leftrightarrow0\in\sigma_p(M(z)-B)$,
and in this case
\[
\dim\ker({A}_B-z)=\dim\ker(M(z)-B);
\]
\item[(ii)] $z\in\sigma_r({A}_B)\Leftrightarrow0\in\sigma_r(M(z)-B)$;
\item[(iii)] $z\in\sigma_c({A}_B)\Leftrightarrow0\in\sigma_c(M(z)-B)$.
\end{itemize}
\end{proposition}

\begin{proof}
(see \cite{GG72A}, \cite{Koc80}). (i) Let $z\in\sigma_p({A}_B)$, and let
${A}_Bf=zf$ for some $f\ne 0$. Then $f\in\sN_z$ and $M(z)\Gamma_0f=\Gamma_1f$.
Taking into account that $B\Gamma_0f=\Gamma_1f$ we have
\[
(M(z)-B)\Gamma_0f=0 \quad \mbox{ and } \quad \Gamma_0{f}\ne 0,
\]
and hence $0\in\sigma_p(M(z)-B)$.

Conversely, if $(M(z)-B)h=0$ for some $h\in\cH\setminus\{0\}$,
then, choosing an $\wt{f}$ such that $\Gamma_0\wt{f}=h$, $\Gamma_1\wt{f}=Bh$,
we obtain
\[
M(z)\Gamma_0\wt{f}=\Gamma_1\wt{f}\quad\mbox{and}\quad
\Gamma_0\wt{f}=h\ne 0.
\]
Thus, $\wt{f}\in\dom(A)\dotplus\sN_z$. Let $f$ be a component of the vector
$\wt{f}$ lying in $\sN_z$. Then ${A}_Bf=zf$ and $f\ne 0$.
It follows that $z\in\sigma_p({A}_B)$.

Finally, multiplicities of the eigenvalues $z\in\sigma_p({A}_B)$ and
$0\in\sigma_p(M(z)-B)$ coincide since the mapping $f\mapsto\Gamma_0f$
sets up a one-to-one correspondence between eigenspaces $\ker({A}_B-z)$ and
$\ker(M(z)-B)$.

(ii) If $z\in\sigma_r({A}_B)$, then $\bar{z}\in\sigma_p(\wt{A}_{B^*})$,
 and from (i) it follows that $0\in\sigma_p(M(\bar{z})-B^*)$.
Since $M(\bar{z})=M^*(z)$, we conclude that $0\in\sigma_r(M(z)-B)$.

(iii) To prove (iii), it suffices to establish the equivalence
\[
z\in\rho(A_B)\Longleftrightarrow 0\in\rho(M(z)-B).
\]
Let us prove the solvability of the equation
\begin{equation}\label{DM_24}
({A}_B-z)\varphi=h
\end{equation}
for any $h\in H$.
We will seek its solution in the form $\varphi=f+g$ assuming that $f\in\sN_z$ and
$g$ is a unique solution to the equation $(A^*-z)g=h$ such that $\Gamma_0g=0$ (i.e.,  $g=(\wt{A}_0-z)^{-1}h$). Since $0\in\rho(M(z)-B)$, there exists a $u_0\in \cH$ such that
$(M(z)-B)u_0=-\Gamma_1g$. Define an $\wt{f}\in\dom(A^*)$ from the following conditions:
\[
\Gamma_0\wt{f}=u_0, \quad\Gamma_1\wt{f}=M(z)u_0.
\]
Then $\wt{f}\in\dom(A)+\sN_z$ and $\wt{f}=f_A+f$ $(f_A\in\dom(A),\
f\in\sN_z)$. It is clear that
\begin{equation}\label{DM_25}
(M(z)-B)\Gamma_0\wt{f}=-\Gamma_1g,\qquad
\Gamma_1(f+g)=B\Gamma_0f=B\Gamma_0(f+g).
\end{equation}
Therefore, $f+g\in\dom({A}_B)$, and the solvability of the equation~\eqref{DM_24}
is proved.

To prove the converse, we will show the solvability of the equation $(M(z)-B)u_0=u_1$ for any  $u_1\in H$ assuming that $z\in\rho({A}_B)$. By definition of a boundary triplet, there exists a $g\in\dom(A^*)$ such that
\[
\Gamma_0g=0,\quad \Gamma_1g=-u_1.
\]
Let
$\varphi=({A}_B-z)^{-1}(A^*-z)g\in\dom({A}_B)$. Since
$f=\varphi-g\in\sN_z$, we obtain that $M(z)\Gamma_0f=\Gamma_1f$ and hence
\[
(M(z)-B)\Gamma_0f=(\Gamma_1-B\Gamma_0)f=(\Gamma_1-B\Gamma_0)(\varphi-g)=-\Gamma_1g=u_1.
\]
Putting $u_0=\Gamma_0f$ we arrive at the desired result.
\end{proof}
   \begin{remark}\label{R:DM_5}
For dissipative extensions ${A}_B$, Proposition~\ref{P:DM_7} was proved before. Namely,
in the case of a minimal Sturm--Liouville operator with an operator potential it was proved
in~\cite{GG72A}, and in the case of a Hermitian operator it was proved in~\cite{Koc80}.
     \end{remark}
\begin{corollary}\label{C:DM_7}
Suppose that $\{\cH,\Gamma_0,\Gamma_1\}$ is a boundary triplet for $A^*$,
$B\in[\cH]$, ${A}_B\supset A$. Then the following conditions are equivalent:
\begin{itemize}
\item[(i)] $z\in\rho({A}_B)$;
\item[(ii)] $\Gamma_1-B\Gamma_0$ isomorphically maps $\sN_z$ onto $\cH$;
\item[(iii)] the extensions $\wt{A}_z$ and ${A}_B$ are transversal.
\end{itemize}
\end{corollary}

\begin{proof}
Since $z\in\rho({A}_B)\Leftrightarrow0\in\rho(M(z)-B)$, the equivalence
(i)$\Leftrightarrow$(ii) is implied by both the obvious relation $(\Gamma_1-B\Gamma_0)|_{\sN_z}=(M(z)-B)\Gamma_0|_{\sN_z}$ and Remark~\ref{R:DM_3}.

To prove the equivalence (i)$\Leftrightarrow$(iii), we notice that,
in view of Proposition~\ref{P:DM_2},
\[
\dom(A^*)=\ker(\Gamma_1-B\Gamma_0)+\ker\Gamma_0.
\]
Therefore,
\begin{eqnarray*}
(\Gamma_1-B\Gamma_0)\dom(A^*)&=&(\Gamma_1-B\Gamma_0)(\dom({A}_B)+\dom({A}_0))=\\
&=&(\Gamma_1-B\Gamma_0)\dom({A}_0)=\Gamma_1(\dom({A}_0))=\cH.
\end{eqnarray*}

Further, if ${A}_0$ and ${A}_B$ are transversal, then
\[
\cH=(\Gamma_1-B\Gamma_0)\dom(A^*)=(\Gamma_1-B\Gamma_0)\sN_z=(M(z)-B)\Gamma_0\sN_z.
\]
It follows that $(M(z)-B)^{-1}\in[\cH]$ since
\[
\ker(M(z)-B)=\{0\}\Leftrightarrow\dom({A}_B)\cap\dom(\wt{A}_z)=\dom(A).
\]
Conversely, if $z\in\rho({A}_B)$, then
\[
(\Gamma_1-B\Gamma_0)\sN_z=\cH=(\Gamma_1-B\Gamma_0)\dom(A^*)\quad\mbox{
and }\quad \dom(A^*)=\dom({A}_B)\dot{+}\sN_z.
\]
\end{proof}

The following Lemma can easily be extracted from the proof of Theorem~\ref{T:DM_1}.

\begin{lemma}\label{L:DM_1}
Suppose that a proper extension $\wt{A}$ of an operator $A$ at some boundary triplet $\{\cH, \Gamma_0,\Gamma_1\}$ has the form $\wt{A}={A}_B$,
where $B\in\cC(\cH^1)$ $(B\in[\cH^1])$. Then at a boundary triplet $\{\cH^1, \Gamma_0^1,\Gamma_1^1\}$ such that $\ker\Gamma_0=\ker\Gamma_0^1$,
the extension $\wt{A}$ is also of the form $\wt{A}={A}_{B_1}$, where $B_1\in\cC(\cH)$ $(B_1\in[\cH])$.
\end{lemma}

\begin{proof}
Formula \eqref{DM_18} gives the following relations between the
boundary triplets $\{\cH,\Gamma_0,\Gamma_1\}$ and $\{\cH^1, \Gamma_0^1,\Gamma_1^1\}$:
\[
\Gamma_0^1=UX_{22}\Gamma_0,\qquad
\Gamma_1^1=UX_{11}(\Gamma_1+K\Gamma_0),
\]
where $U$ is a unitary operator from $\cH$ to $\cH^1$,
$K=X_{11}^{-1}X_{12}=K^*\in[\cH]$. The relation $\Gamma_1=B\Gamma_0$ yields
$\Gamma_1=UX_{11}(B+K)\Gamma_0$. By putting
\begin{equation}\label{DM_26}
B_1=UX_{11}(B+K)X_{11}^*U^*
\end{equation}
one obtains
$\dom(\wt{A})=\dom({A}_{B_1})=\ker(\Gamma_1^1-B_1\Gamma_0^2)$.
This completes the proof.
\end{proof}

In what follows $\sS(H)$ stands for a two-sided ideal in the algebra
$[\cH]$.

\begin{proposition}\label{P:DM_8}
Let $\{\cH, \Gamma_0,\Gamma_1\}$ be a boundary triplet of
$B_1\in\cC(\cH), B_2\in\cC(\cH)$, and let ${A}_{B_1}$ and $ A_{B_2}$ be almost
solvable extensions of the operator $A$ with a common regular point
$z\in\rho({A}_{B_1})\cap\rho({A}_{B_2})$. Then
\begin{equation}\label{eq:ResComp}
    ({A}_{B_1}-z)^{-1}-({A}_{B_2}-z)^{-1}\in\sS(H)\Longleftrightarrow
(B_1-M((z))^{-1}-(B_2-M(z))^{-1}\in\sS(\cH).
\end{equation}
\end{proposition}
\begin{proof} {\it Step 1.}
First we prove Proposition~\ref{P:DM_8} assuming that the boundary triplet
$\{\cH, \Gamma_0,\Gamma_1\}$ is the canonical one\footnote[1]{For $z\neq i$
the "canonical"\ boundary triplet is constructed as follows:
$$\cH^0=\sN_{z},\quad \Gamma_0^0=P_z-VP_{\bar{z}},\quad
\Gamma_1^0=zP_z-\bar{z}VP_{\bar{z}},$$ where $V$ is an isometry from
$\sN_{\bar z}$ to $\sN_{ z}$ such that
$\dom(\widetilde{A}_0)=\dom(A)\dotplus(I+V)\mathfrak{N}_z $, and
$P_{z}, P_{\bar z}$ are projectors in $\dom(A^*)$ onto
$\mathfrak{N}_{z}$ and $\mathfrak{N}_{\bar z}$ respectively in the decomposition
$\dom(A^*)=\dom(A)\dotplus\mathfrak{N}_{z}\dotplus
\mathfrak{N}_{\bar z}$.} $\{\cH^0, \Gamma_0^0,\Gamma_1^0\}$. Since
\[
[({A}_{B_1}-z)^{-1}-({A}_{B_2}-z)^{-1}]|_{\sN_z}=0,
\]
it follows that
\begin{equation}\label{eq:Res2}
({A}_{B_1}-z)^{-1}-({A}_{B_2}-z)^{-1}\in\sS(H)\Longleftrightarrow
[({A}_{B_1}-z)^{-1}-({A}_{B_2}-z)^{-1}]|_{\sN_z}\in\sS(\sN_z).
\end{equation}

Hence it suffices to consider the difference of resolvents
on the subspace $\sN_{\bar{z}}$. To this purpose, we represent the vector
$g=g_{\bar{z}}+g_z\in\dom({A}_B)\cap(\sN_z\dot{+}\sN_{\bar{z}})$ as
\[
g=f_0+f_z\quad\mbox{ by putting  }\quad
f_0=(I+V)g_{\bar{z}}\in\dom({A}_0),\quad
f_z=g_z-Vg_{\bar{z}}\in\sN_{\bar{z}}
\]
(here $V$ is an isometry from $\sN_{\bar{z}}$ onto $\sN_z$). Then
\[\begin{array}{ll}
  \Gamma_0^0g& =\Gamma_0^0(f_0+f_z)=\Gamma_0^0f_z=(P_z-VP_{\bar{z}})f_z=f_z,  \\
  \Gamma_1^0g &
  =(zP_z-\bar{z}VP_{\bar{z}})(g_z+g_{\bar{z}})=zf_z+(z-\bar{z})Vg_{\bar{z}}.
\end{array}\]
Therefore, the relation $\Gamma_1^0g=B\Gamma_0^0g$ valid for any $g\in\dom({A}_B)$
is equivalent to
\[
(B-z)f_z=(z-\bar{z})Vg_{\bar{z}}.
\]
But in this case $\ker(B-z)=\{0\}$, $(B-z)^{-1}\in\cC(\cH)$ and
\[
f_z=(z-\bar{z})(B-z)^{-1}Vg_{\bar{z}}.
\]
On the other hand,
\[
({A}_B-z)g=({A}_B-z)(g_z+g_{\bar{z}})=(\bar{z}-z)g_{\bar{z}}\quad
\mbox{and}\quad ({A}_0-z)f_0=({A}_0-z)(I+V)g_{\bar{z}}=(\bar{z}-z)g_{\bar{z}}.
\]
Hence we have
$f_z=g-f_0=(z-\bar{z})[(\wt{A}-z)^{-1}-({A}_B-z)^{-1}]g_{\bar{z}}$.
It follows that
\begin{equation}\label{DM_29}
(B-z)^{-1}V=[({A}_0-z)^{-1}-({A}_B-z)^{-1}]|_{\sN_{\bar{z}}}.
\end{equation}
In view of~\eqref{eq:Res2}, this proves the equivalence~\eqref{eq:ResComp}.

 {\it Step 2.}
In the case when a boundary triplet $\{\cH, \Gamma_0,\Gamma_1\}$ is arbitrary
we consider the canonical boundary triplet of form~\eqref{DM_6} such that  $\ker\Gamma_0^0=\ker\Gamma_0=\dom({A}_0)$. By Lemma~\ref{L:DM_1}, the extensions
${A}_{B_j}$ in the boundary triplet $\{\cH^0, \Gamma_0^0,\Gamma_1^0\}$
are of the form ${A}_{B_j^0}$, where ${B}_j^0\in\cC(\cH^0)$ $(j=1,2)$.
It follows from~\eqref{DM_26}  that
\begin{equation}\label{DM_27}
B_j=UX_{11}(B_j^0+K)U^*X_{11}^*U^*\qquad (j=1,2),
\end{equation}
where $X_{11}$ is an automorphism in $\cH^0$ and $U$ is an isometry from $\cH^0$
to $\cH$. The Weyl functions $M(z)$ and $M_0(z)$ corresponding to the boundary
triplets $\{\cH, \Gamma_0,\Gamma_1\}$ and $\{\cH^0,\Gamma_0^0,\Gamma_1^0\}$
respectively, are related by
\begin{equation}\label{DM_28}
M(z)=UX_{11}(M_0(z)+K)X_{11}^*U^*.
\end{equation}
From both \eqref{DM_27} and \eqref{DM_28} we obtain that
\begin{equation}\label{DM_30}
(B_j^0-M_0(z))^{-1}=U^*X_{11}^*(B_j-M(z))^{-1}X_{11}U.
\end{equation}

Since $M_0(z)=z$, then, as was proved in Step 1,  the following  equivalence holds:
\begin{equation}\label{eq:ResCompA}
    ({A}_{B_1}-z)^{-1}-({A}_{B_2}-z)^{-1}\in\sS(H)\Longleftrightarrow
(B_1^0-M_0((z))^{-1}-(B_2^0-M_0(z))^{-1}\in\sS(\cH).
\end{equation}
Thus, the equivalence \eqref{eq:ResComp} follows from \eqref{eq:ResCompA}
and~\eqref{DM_30}.
\end{proof}

\begin{theorem}\label{T:DM_2}
Suppose that, at some boundary triplet $\{\cH, \Gamma_0,\Gamma_1\}$, the operators
${A}_{B_1}$ and ${A}_{B_2}$ $(B_1,B_2\in\cC(\cH))$
have a common regular point $z\in\rho({A}_{B_1})\cap\rho({A}_{B_2})$. If
$\rho(B_1)\cap\rho(B_2)\neq\emptyset$, then for any $\zeta\in\rho(B_1)\cap\rho(B_2)$
the following relation holds:
\[
({A}_{B_1}-z)^{-1}-({A}_{B_2}-z)^{-1}\in\sS(H) \Longleftrightarrow
 (B_1-\zeta)^{-1}-(B_2-\zeta)^{-1}\in\sS(\cH).
\]
\end{theorem}

\begin{proof}
It suffices to establish the equivalence
\[
R_{B_1}(\zeta)-R_{B_2}(\zeta)\in\sS(\cH) \Longleftrightarrow (B_1-M(z))^{-1}
-(B_2-M(z))^{-1}\in\sS(\cH).
\]
To this aim, we use the identities
\begin{equation}\label{DM_31}
[I+(\zeta-M(z))(B-\zeta)^{-1}]^{-1}=I+(M(z)-\zeta)(B-M_z)^{-1},
\end{equation}
\begin{equation}\label{DM_32}
[I+(B-\zeta)^{-1}(\zeta-M(z))]^{-1}=I+(B-M(z))^{-1}(M(z)-\zeta),
\end{equation}
that hold for any $z\in\rho({A}_B)$ and $\zeta\in\rho(B)$. Recall that,
by Proposition~\ref{P:DM_7},
\[
z\in\rho({A}_B)\Leftrightarrow(B-M(z))^{-1}\in[\cH].
\]
In particular, the identities~\eqref{DM_31}, \eqref{DM_32} ensure
the bounded invertibility of the operators $I+(\zeta-M(z))(B-\zeta)^{-1}$ and
$I+(B-\zeta)^{-1}(\zeta-M(z))$ and easily follow from the Hilbert
identity for resolvents. Further on,
\begin{eqnarray*}
(B_2-M(z))^{-1}&-&(B_1-M(z))^{-1}=\\&=&R_{B_2}(\zeta)[I+(\zeta-M(z))R_{B_2}(\zeta)]^{-1}
-[I+R_{B_1}(\zeta)(\zeta-M(z))]^{-1}R_{B_1}(\zeta)\\
&=&[I+R_{B_1}(\zeta)(\zeta-M(z))]^{-1}[R_{B_2}(\zeta)-R_{B_1}(\zeta)]
[I+(\zeta-M(z))R_{B_2}(\zeta)]^{-1}.
\end{eqnarray*}
Due to the identities \eqref{DM_31}, \eqref{DM_32} this completes the proof.
\end{proof}

\begin{corollary}\label{C:DM_8}
Suppose that under the assumptions of Theorem~\ref{T:DM_2} $B_1,B_2\in[\cH]$.
Then
\[
({A}_{B_1}-z)^{-1}-({A}_{B_2}-z)^{-1}\in\sS(H)  \Longleftrightarrow
B_1-B_2\in\sS(\cH).
\]
\end{corollary}

The proof follows from the identity
\[
R_{B_1}(z)-R_{B_2}(z)=R_{B_1}(z)[B_2-B_1]R_{B_2}(z).
\]

In the following corollary a class of extensions with
a discrete spectrum, which have the same principal terms in the asymptotic behavior
of the $s$--numbers, will be selected.

\begin{corollary}\label{C:DM_9}
Suppose that under the assumptions of Theorem~\ref{T:DM_2}
$({A}_{B_1}-z_0)^{-1}\in\sS_\infty(H)$
and
\[
\lim\limits_{n\to\infty}n^\alpha
s_n(({A}_{B_1}-z_0)^{-1})=a\quad\mbox{ for some }\quad
\alpha>0, a>0.
\]
Then for the validity of the relation
\[
\lim\limits_{n\to\infty}n^\alpha s_n(({A}_{B_2}-z_0)^{-1})=a
\]
it suffices that
\[
\lim\limits_{n\to\infty}n^\alpha s_n(R_{B_1}(z_0)-R_{B_2}(z_0))=0.
\]
\end{corollary}

The proof is implied by both the Ky Fan lemma (\cite[Theorem 2.2.3]{GohKr65}) and the relation
\begin{equation}\label{DM_33}
[({A}_{B_1}-z_0)^{-1}-({A}_{B_2}-z_0)^{-1}]|_{\sN_{\bar{z}_0}}=T_1[R_{B_1}(z)-R_{B_2}(z)]T_2^*
\end{equation}
in which $T_1$ and $T_2$ are isomorphisms from $\sN_z$ to $\sN_{\bar{z}}$.
$\Box$

In the following corollary, for a given extension of ${A}_{B_1}$
with a discrete spectrum, extensions ${A}_{B_2}$ with a more thick spectrum (i.e., $\lim\limits_{n\to\infty}s_n({A}_{B_2})/s_n({A}_{B_1})=0$) are constructed.
\begin{corollary}\label{C:DM_11}
Suppose that, in the assumptions of Theorem~\ref{T:DM_2},
$({A}_{B_1}-z_0)^{-1}\in\sS_\infty(H)$ and
$s_n(({A}_{B_1}-z_0)^{-1})\sim a/n^\alpha$, $a>0$ and
$0<\beta<\alpha$. Then the limits
\[
\lim\limits_{n\to\infty}n^{\beta}s_n(({A}_{B_2}-z_0)^{-1})\quad\mbox{and}\
\quad\lim\limits_{n\to\infty}n^{\beta}s_n(R_{B_1}(z)-R_{B_2}(z))
\]
exist, and are finite and different from zero only simultaneously.
\end{corollary}
\begin{remark}\label{R:DM_6}
The results close to Theorem~\ref{T:DM_2} were obtained
by other technique in the dissipative case in~\cite{Brod69}, \cite{GK78}.
Moreover, Corollaries~\ref{C:DM_8}---\ref{C:DM_11} in particular cases
were obtained just before in~\cite{Geh69}, \cite{Geh82}, \cite{GG72}.
\end{remark}
\section{Extensions of positive operators and Weyl function}\label{4}

\subsection{Positive boundary {triplet}s} In the study of various classes of extensions of a symmetric operator with a real point of regular type
$-a \in \wh{\rho}(A)$ (for instance semi-bounded operator) M.I.~Vishik \cite{Vi52} and M.Sh.~Birman \cite{Bir56} were using  the following decomposition of $\dom(A^*)$:
\[
\dom(A^*) = \dom({A}_0) \dotplus \mathfrak {N}_{-a},\quad -a \in
\wh{\rho}({A}_0)
\]
instead of the J. von Neumann formula. Further development of this approach has led to the concept of
positive boundary {triplet} (see~\cite{Koc79}), the use of which is
very convenient in the study of  proper extensions of the operator ${A}$.

\begin{definition} \label{D:DM_6} {\rm(\cite{Koc79, Mikh80})} Let $0 \in
\wh{\rho}(A)$. A boundary {triplet} $\{\mathcal{H}, \Gamma_0,
\Gamma_1\}$ of the operator $A^*$ is called a positive boundary {triplet},
corresponding  to the extension ${A_0}={A}^*_0$, if
$$(A^*f, g)=({A}_0f_0, g_0)+(\Gamma_1f, \Gamma_0g)_{\mathcal{H}},$$
where $f,g\in\dom(A^*)$ and $f_0, g_0 \in \dom({A}_0)$ are components of the vectors $f, g$ in
the decomposition
$$\dom(A^*) = \dom({A}_0) \dotplus \ker (A^*).$$
\end{definition}

If $-a \in {\rho}({A}_0)$, then the positive boundary {triplet}
$\{\mathcal{H}, \Gamma_0^a, \Gamma_1^a \}$ for the operator  $A^* +
a$ can be constructed (see~\cite{Koc79, Mikh80}) by
\begin{equation}\label{DM_34}
\mathcal{H}_a = \mathfrak{N}_{-a},\quad \Gamma_1^a =
P(-a)({A}_0+a)P_1,\quad \Gamma_0^a = P_0,
\end{equation}
where $P(-a)$ is the orthogonal projection from $H$ to $\mathfrak{N}_{-a}$;
and $P_1$, $P_0$ are skew projections from
\begin{equation}\label{eq:domA*}
    \dom(A^*) = \dom({A}_0) \dotplus \mathfrak {N}_{-a}
\end{equation}
onto $\dom({A}_0)$  and  ${\sN}_{-a}$, respectively.

\begin{proposition}\label{P:DM_9} Let $\{\mathcal{H}, \Gamma_0, \Gamma_1\}$ be a  boundary {triplet} for $A^*$,
 $B\in \mathcal{C}(\mathcal{H})$, $-a \in
\rho({A}_B)\cap\rho({A}_0)$. Then the following relation
holds
\eq{DM_35}{[({A}_B + a) - ({A}_0 + a)] \mid_{\mathfrak{N}_{-a}}
= T(-a)^*(B - M(-a))^{-1}T(-a),} where $T(-a) \in
[\mathfrak{N}_{-a}, \mathcal{H}]$, and $T(-a)^{-1} \in [\mathcal{H},
\mathfrak{N}_{-a}]$.
\end{proposition}

\begin{proof}
Alongside with  the boundary {triplet}  $\{\mathcal{H},
\Gamma_0, \Gamma_1 \}$ consider the boundary {triplet} of the form
$\eqref{DM_34}$ and the corresponding Weyl functions $M(z)$ и $M_a(z)$.
By Lemma~\ref{L:DM_1}, the domain of the operator
 $\wt A:={A}_B$ has the following form
\[
\dom(\wt A)=\ker{(\Gamma_1^a-B_a\Gamma_0^a)},
\]
 where $B_a \in \mathcal C(\mathfrak{N}_{-a})$. Since $M_a(-a)=0$,
 then by  Proposition~\ref{P:DM_7} and the condition $-a \in
\rho({A}_B)$, one obtains $0 \in \rho(B_a) = \rho(B_a - M_a(-a))$.
 Let us show that
\eq{DM_36}{B^{-1}_a = [({A}_B + a)^{-1} - ({A}_0 +
a)^{-1}]\mid_{\mathfrak{N}_{-a}}.}

If $f \in \dom({A}_B) \subset\dom(A^*)$, then in accordance with
the decomposition~\eqref{eq:domA*},
\[
 f = f_0 +
f_{-a},\quad\mbox{ where  }\quad f_0 \in \dom({A}_0),\quad f_{-a}
\in \mathfrak{N}_{-a}.
\]
Hence $ \Gamma_0^{a}f = P_0f = f_{-a}$.  Further,
\[
P_1f = f_{0}, \quad({A}_0 + a)f_0= g = g_1 + g_{-a},\quad\mbox{
where }\quad g_1 \in \ran(A + a),\quad g_{-a} \in \mathfrak{N}_{-a}
\] and, therefore,
\[
\Gamma_1^af = P({A}_0 + a)P_1f = P(g_1 + g_{-a}) = g_{-a}.
\]
Consequently, the equality $\Gamma_1^af = B_a\Gamma_0^af$ takes the
form $B_af_{-a} = g_{-a} $, which is equivalent to the relation
\eq{DM_37}{f_{-a} = B_{a}^{-1}g_{-a} = (B_a - M_a(-a))^{-1}f_{-a}.}
On the other hand
\[
({A}_B + a)f = (A^* + a)(f_0 + f_{-a}) = ({A}_0
+ a)f_0 = g,
\]
whence \eq{DM_38}{f_{-a} = f - f_0 = [({A}_B + a)^{-1} - ({A}_0 +
a)^{-1}]g = [({A}_B + a)^{-1} - ({A}_0 + a)^{-1}]g_{-a}.}

Comparing the formulas $\eqref{DM_37}$ and $\eqref{DM_38}$, we
arrive at the relation $\eqref{DM_36}$. The equality $\eqref{DM_35}$
is implied by $\eqref{DM_36}$ and the following relation between
the Weyl functions $M(z)$, $M_a(z)$ and operators $B$, $B_a$:
$$M(z) = T(M_a(z)+K)T^*, \quad
B = T(B_a + K)T^*,$$ where $K = K^* \in [\mathfrak{N}_{-a}]$, $T =
UX_1$. This completes the proof.
\end{proof}

\begin{remark}\label{R:DM_7}
Notice, that when  $a = 0$ the positive boundary {triplet}
 $\{\mathfrak{N}_0, \Gamma_0^0, \Gamma_1^0 \}$ of the form $\eqref{DM_34}$ is obtained from the
 "canonical" boundary {triplet}s $\{\mathcal{H}(z), \Gamma_0(z), \Gamma_1(z) \}$,
 where
 \[
 \mathcal{H}(z) = \mathfrak{N}_z, \quad
 \Gamma_1(z) = zP_z - \overline{z}VP_{\overline{z}},\quad
  \Gamma_0(z) = P_z - VP_{\overline{z}},
 \]
 by "limiting process" as $z \to 0$.
 We omit the cumbersome calculations leading to this and to a bit more general relation
 $\Gamma(z) \to \Gamma_1^a - a\Gamma_0^a, \Gamma_0(z) \to \Gamma_0^a$, as $z \to -a$,
  but we  merely point out, that the formula $\eqref{DM_36}$ follows easily from \eqref{DM_27},
  because
\[
V(z) = -({A}_0 - \bar{z})({A}_0 - z)^{-1}
\mid_{\mathfrak{N}_{\bar{z}}} =
  -I + (\bar{z} - {z})({A}_0 - z)^{-1} \mid_{\mathfrak{N}_{\bar{z}}} \quad \to -I
  \quad\mbox{where
}\quad z \to -a
\]
  and $B(z) \to B_a - aI$, where $z \to -a$.
\end{remark}

\begin{remark}\label{R:DM_8}
Applying the above "limiting process" to the formula~\eqref{DM_20}, one
obtains  the following representation of the Weyl function $M(z)$ corresponding to the boundary {triplet}
$\eqref{DM_34}$

\eq{DM_39}{M(z) = (z+a)P[I + (z + a)({A}_0 - z)^{-1}]|_{\sN_{-a}}.}
However, the direct proof is even shorter. Indeed, writing  $f_z \in
\mathfrak{N}_z$ as
\[
f_z = ({A}_0 + a)({A}_0 - z)^{-1}f_{-a}, \quad
f_{-a} \in
\mathfrak{N}_{-a},
\]
one obtains the equalities  $\Gamma_0f_z = f_{-a}$, $\Gamma_1f_z = (z + a)P({A}_0 +
a)({A}_0)^{-1}f_{-a}$, which lead to~$\eqref{DM_39}$.

Notice also that   the equality
\eqref{DM_21} for $M(z)$ easily follows from $\eqref{DM_39}$, if the $\gamma$-field  is defined as follows
\[
\gamma(-a) = I_{\mathfrak{N}_{-a}}, \quad
\gamma(z) = ({A}_0 +
a)({A}_0 - z)^{-1}\gamma(-a).
\]
\end{remark}

\subsection{Stieltjes class}
Let $A\ge 0$ be a nonnegative symmetric operator in $\frak H$.
Recall that in the set of non-negative self-adjoint extensions of $A$
there exists (see \cite{Kr47}) the maximal and the minimal extensions $\wt{A}_{F}$
and $\wt{A}_{K}$,  called the Friedrichs and the Krein
extensions of $A$, respectively. These extensions are characterized by the inequalities
  \begin{equation}\label{DM_Fridr_Krein}
(\wt{A}_{F} + x)^{-1} \le(\wt{A} + x)^{-1} \le (\wt{A}_{K} + x)^{-1}, \quad x >0,
    \end{equation}
in which $\wt{A}\ge 0$ is an arbitrary non-negative self-adjoint
extension of the operator $A\ge 0$.

\begin{definition}\label{D:DM_7}{\rm(\cite{GG72})}.
An operator function $F(z)$ holomorphic on the complex plane with a cut along the half-line
$[0,\infty)$  with values in
$[\mathcal{H}]$ is called the Stieltjes function, if it belongs to the class $(R)$ and $F(x) \ge 0$ for all $x <0$.
\end{definition}

The class of Stieltjes operator functions is denoted by $(S)$. Recall that $F \in (S)$,
if $F \in (R)$ and $zF(z) \in (R)$ (see~\cite{KrNud}). We also write: $F_1 \in
(\widehat{S})$, if $F_1(z) = F(z) + K,$ where  $F \in (S)$ and $K = K^* \in [\mathcal{H}]$.

\begin{theorem}\label{T:DM_3}
Let $M(z)$ be a Weyl function corresponding to a boundary {triplet}
$\{\mathcal{H}, \Gamma_0, \Gamma_1\}.$ Then $M \in
(\widehat{S})$, if and only if the operator  ${A}_0$ is positive and transversal to the Friedrichs extension $\wt{A}_{F}$.
\end{theorem}
\begin{proof}
Notice first that the operator ${A}_0$ is positive.
Hence the operators ${A}_0$ and $\wt{A}_{F}$ are transversal to the operator $\wt{A}_{-a}$ and by Corollary~\ref{C:DM_2}, there exist boundary {triplet}s
$\{\mathcal{H}, \Gamma_0, \Gamma_1^{a}\}$,
$\{\mathcal{H}, \Gamma_0^{F}, \Gamma_1^a\}$, such that
\[
\dom({A}_0) = \ker\Gamma_0,\quad \dom(\wt{A}_{F}) =
\ker\Gamma_0^{F},\quad\dom(\wt{A}_{-a}) = \ker\Gamma_1^{a}.
\]
In view of Proposition~\ref{P:DM_1} the corresponding  Weyl function are connected by the equality
\eq{DM_40}{M^{-1}_{F}(z) = \wt{M}^{-1}(z) + B.}
It follows from the formula $\eqref{DM_39}$  and the first of the inequalities
\eqref{DM_Fridr_Krein}
  \[
(\wt{A}_{F} + x)^{-1} \le({A}_0 + x)^{-1} , \quad x >0,
    \]
that $B \ge 0$. We note that transversality of the operators
${A}_0$ and $\wt{A}_{F}$ is equivalent to the existence of a bounded inverse of $B$ (see Corollary~\ref{C:DM_1}).

It follows from $\eqref{DM_40}$  that $M_{F}(z)[\wt{M}(z)^{-1}
+ B] = I$. The operator-function $\wt{M}(-x)^{-1}$ increases monotonically with $x > 0$, what follows from the formula $\eqref{DM_39}$. Furthermore,
\[
 \lim\limits_{x \to
\infty}(M_{F}(-x)h, h) = - \infty, \quad \mbox{for all}\quad h \in
\mathcal{H}\setminus \{0\}
\] (see~\cite{KrOv78}). Hence (see \cite{KrOv77}) there exists
$s-\lim\limits_{x  \uparrow \infty} [\wt{M}(-x)^{-1} + B] = 0$, i.e.
\eq{DM_41}{s-\lim\limits_{x \uparrow +\infty} \wt{M}(-x)^{-1} = -B.}

Since $\ker B = 0$, then the condition $\eqref{DM_41}$ means that
$\wt{M}(-x)$ converges strongly to  $-B^{-1}$ in the generalized sense (see~\cite{Kato}).

Let${A}_0$ and $\wt{A}_{F}$ be transversal.  Then $B^{-1} \in [\mathcal{H}]$  and
$-B^{-1} = s-\lim\limits_{x \uparrow \infty} \wt{M}(-x)$. By
\eqref{DM_40}, $\wt{M}(-x) +B \ge 0$ and $\wt{M}(z) \in (\widehat{S})$.

If the extensions ${A}_0$ and $\wt{A}_{F}$ are not transversal, then  the operator $B^{-1}$
is unbounded.  Applying the theorem on the semi-continuity of the spectrum below (see~\cite{Kato}),
we obtain for all $\lambda \in \sigma(B^{-1})$ that any interval $(-\lambda, -\lambda +
\varepsilon)$ $(\varepsilon>0)$ contains points of the spectrum of $\wt{M}(-x)$, for all $x$  large enough. Thus, the operator function $\wt{M} \notin (\widehat{S})$.

To complete the proof it remains to use the relation
$M(z) = C\wt{M}(z)C^* + D$. Theorem is proved.
\end{proof}

Note that Theorem~\ref{T:DM_3} can be also deduced from the following proposition, which is of  independent interest.

\begin{proposition}\label{P:DM_10}
Let $Q(x)$ $(x > 0)$ be an operator function with values in the set of positive operators in  $[\mathcal{H}]$, such that
\begin{itemize}
\item[\;\;\rm (i)] $Q(x)$ decreases monotonically on the half-line $(0, \infty)$;
\item[\;\;\rm (ii)] $s-\lim\limits_{x \uparrow \infty} Q(x) = B$;
\item[\;\;\rm (iii)] The operator$B$   is invertible in a wide sense.
\end{itemize}
Then there exists the limit
\[
s-\lim\limits_{x \uparrow \infty}
B^{1/2}Q(x)^{-1}B^{1/2} = I_\cH.
\]
\end{proposition}
\begin{proof}
It follows from the inequality $Q(x) \ge B$, that $B^{1/2}Q(x)^{-1}B^{1/2}
\le I$. The operator function  $B^{1/2}Q(x)^{-1}B^{1/2}$ increases monotonically with $x > 0$, consequently
\begin{equation}\label{eq:BQB}
    s-\lim\limits_{x \uparrow \infty} B^{1/2}Q(x)^{-1}B^{1/2} = K \le I.
\end{equation}
On the other hand, if $x<y$, then we obtain  $Q(x)^{-1} \le Q(y)^{-1}$ and

\eq{DM_42}{I \le Q(x)^{1/2}B^{-1/2}[B^{1/2}Q(y)^{-1}B^{1/2}]B^{-1/2}Q(x)^{1/2}.}

Proceeding to the limit in the inequality $\eqref{DM_42}$  as $y \uparrow \infty$ and using~\eqref{eq:BQB}, we get
\eq{DM_43}
{I \le Q(x)^{1/2}B^{-1/2}KB^{-1/2}Q(x)^{1/2}.}
Let further, $T(x) = K^{1/2}B^{-1/2}Q(x)^{1/2}$. Then the inequality $\eqref{DM_43}$
takes the form $T(x)^*T(x) \ge I$. Clearly,  $\ker(T(x)) = \{0\} =
\ker(T(x)^*)$, $x>0$, consequently  $T(x)^*T(x)$ and $T(x)T(x)^*$ are unitary  equivalent and hence $T(x)T(x)^* \ge I$, i.e.
\[
K^{1/2}B^{-1/2}Q(x)B^{-1/2}K^{1/2} \ge I.
\]
Proceeding in the last inequality to the limit  as $x \uparrow \infty$ and using the condition
(ii), we obtain $K \ge I$. However, inequality $K \le I$ completes the proof.
\end{proof}

\begin{remark}\label{R:DM_9}
 Let us remind the definition of the $Q_{F}$ $(Q_{K})$ function in the terminology of M.G. Krein and  и I.E. Ovcharenko (see~\cite{KrOv78}). Assume, that the operator $A$ is nonnegative, $\wt{A}_{F}$  and $\wt{A}_{K}$ are the Friedrichs and the Krein extensions of the operator $A$ and $\gamma_{K}(z)$ and $\gamma_{F}(z)$ are $\gamma$-fields of the pairs $(A,\wt{A}_{F})$ and $(A,\wt{A}_{K})$, respectively. It is not required that the $\gamma_{F}(z)^{-1}$,
$\gamma_{K}(z)^{-1} \in [\mathfrak{N}_z, \mathcal{H}]$, but it is assumed only that $\gamma_{F}(-a)$ and $\gamma_{K}(-a)$ are single-valued maps
from $\mathcal{H}$ onto $C^{1/2}_a\mathfrak{N}_{-a}$, where
\[
C_a = 2a[(\wt{A}_{K} + a)^{-1} - (\wt{A}_{F} + a)^{-1} ].
\]
An operator function  $Q_{F}(z)$ $(Q_{K}(z))$
holomorphic in  $\Ext[0, \infty)$, which satisfies~\eqref{DM_21}  and the condition
\[
s-\lim\limits_{x \uparrow 0}Q_{F}(x) = 0,\quad (s-\lim\limits_{x
\uparrow -\infty}Q_{K}(x) = 0).
\]
is called the $Q_{F}$ $(Q_{K})$ function of the operator $A$.
 One of function  $Q_{F}(z)$
$(Q_{K}(z))$ takes the form  \eq{DM_44} {\begin{array}{l}
   Q_{F}(z) =
\{-2aI + (z+a)C^{1/2}_a[I + (z+a)(\wt{A}_{F} -z)^{-1}C_a^{1/2}]\}
\mid_{\mathfrak{N}_{-a}}, \\
   Q_{K}(z) = \{2aI + (z+a)C^{1/2}_a[I + (z+a)(\wt{A}_{K} -z)^{-1}C_a^{1/2}]\}
\mid_{\mathfrak{N}_{-a}}.
 \end{array}}

Let $\{\mathcal{H}, \Gamma^{F}_1, \Gamma^{F}_2\}$ and
$\{\mathcal{H}, \Gamma^{K}_1, \Gamma^{K}_0\}$ are positive boundary {triplet}s
for the operator $A^* + a$ of the form $\eqref{DM_34}$, in which ${A}_0 =
\wt{A}_{F}$ and ${A}_0 = \wt{A}_{K}$, respectively,  $M^a_{F}(z)$ and
$M^a_{K}(z)$ are the corresponding Weyl functions. It follows from the relations
$\eqref{DM_39}$ and $\eqref{DM_44}$,  that  the pairs of  functions
$Q_{F}(z), M^a_{F}(z)$ and $Q_{K}(z), M_{K}(z)$ are connected by the equalities
$$Q_{F}(z) = [-2aI + C^{1/2}_aM^a_{F}(z)C_a^{1/2}]\mid_{\mathfrak{N}_{-a}},\quad
Q_{K}(z) = [2aI +
C^{1/2}_aM^a_{K}(z)C_a^{1/2}]\mid_{\mathfrak{N}_{-a}}.
$$

If $\wt{A}_{F}$ and $\wt{A}_{K}$ are transversal, then the operator $C_a$
is invertible  (see Corollary~\ref{cor:DM_3}), and by $\eqref{DM_44}$ we get
$Q_{F}(z)$, $Q_{K}(z)$ are Weyl functions of the operators $\wt{A}_{F}$ and
$\wt{A}_{K}$ for the following choices of  boundary {triplet}s
$$\{\mathcal{H}, -2aC_a^{-1/2}\Gamma_0^{F}+C_a^{1/2}\Gamma_1^{F},
C_a^{-1/2}\Gamma_0^{F}\}, \quad \{\mathcal{H},
2aC_a^{-1/2}\Gamma_0^{M}+C_a^{1/2}\Gamma_1^{M},
C_a^{-1/2}\Gamma_0^{F}\}.$$

Operator $C_a$ is expressed in terms of $M_{F}^a(z)$, $M_{K}^a(z)$: $$C_a =
2a[s-\lim\limits_{x \uparrow 0}(M_{F}^a)(x)^{-1}] =
2a[s-\lim\limits_{x \uparrow 0}(M_{K}^a)(x)^{-1}].$$

These relations are true without assuming transversality of $\wt{A}_{F}$ and
$\wt{A}_{K}$.
\end{remark}

\subsection{A criterion for finiteness of the negative spectrum}
 Let $\{\mathcal{H}, \Gamma_0,
\Gamma_1\}$ be a boundary {triplet}, such that  ${A}_0 \ge 0$. In this case the Weyl function
$M(z)$ of the operator ${A}_0$, corresponding to this boundary {triplet}, is defined and holomorphic on
 $(-\infty, 0)$. Being an  $R$-function $M(x)$  increases monotonically on
 $(-\infty, 0)$, and if extensions $A_0$ and $\wt{A}_{K}$  are disjoint, then the following equality
  defines a self-adjoint operator
\begin{equation}\label{eq:M0}
    M(0) := s-R-\lim\limits_{x \uparrow 0}M(x),
\end{equation}
as a strong resolvent limit of operators $M(x)$ at $x\to 0$
(see~\cite{Kato}).
If the expansions  $A_0$ and $\wt{A}_{K}$ are transversal, then the operator
 $M(0)$ is bounded, $M(0)\in[H]$.

\begin{proposition}\label{P:DM_11}
Let $\{\mathcal{H}, \Gamma_0, \Gamma_1\}$ be a boundary {triplet} for the operator $A^*$, where ${A}_0$ is
a nonnegative extension of $A$ transversal  to the Krein's extension $\wt{A}_{K}$.  Let
 $B = B^* \in \mathcal{C}(\mathcal{H})$ and  let $A_B$ be a proper extensions of the operator
$A$, which is defined by the equality
 $\dom({A}_B) = \ker(\Gamma_1 - B\Gamma_0)$.
In order that the negative part of the spectrum of the operator ${A}_B$:
\begin{itemize}
\item[\;\;\rm (a)] to consist of $n$ points ($0 \le n \le \infty$);
\item[\;\;\rm (b)] to have a unique accumulation point $0$,
\end{itemize}
it is sufficient, and if ${A}_0 = \wt{A}_{F} $, then also necessary, that the operator
$B -M(0)$ to have the same property.
\end{proposition}
\begin{proof}
In view of $\eqref{DM_35}$ the following formula holds
\eq{DM_45}
{[({A}_B - x)^{-1} - ({A}_0 - x)^{-1}] \mid_{\mathfrak{N}_x}=
T(x)^*(B-M(x))^{-1}T(x),}
for all $x \in \rho({A}_B) \cap (-\infty,0)$. Assume that $\dim E_{B - M(0)}(-\infty, 0) = n <
\infty$. Since\footnote[1]{If the operator function $T(x)$ is not monotonic then  $\dim
E_{T(x)}(-\infty, 0)$, can increase, as it can be seen from elementary examples.} the function $B - M(x)$ is monotonically decreasing on $(-\infty,0)$, then $\dim E_{B -
M(-\varepsilon )}(-\infty, 0) = n $ for all $\varepsilon >0$ small enough. In view of~\eqref{DM_45}, при каждом
$x <0$ the operator $[({A}_B +\varepsilon)^{-1} - ({A}_0 +\varepsilon)^{-1}]$ also has $n$ negative eigen-values with account of multiplicity for all $\varepsilon >0$ small enough. Then, as follows from the results of papers
\cite{Bir56, Kr47, Mikh80}, the operator  ${A}_B +\varepsilon I$ has the same property for all  $\varepsilon >0$, and hence also the operator ${A}_B$.

Let ${A}_0 = \wt{A}_{F}$ and $\dim E_{{A}_{B }}(-\infty, 0) = n$. Then $\dim E_{{A}_{B + \varepsilon I}}(-\infty, 0) = n$ for all $\varepsilon$ small enough and, in view of~\cite{Bir56, Kr47, Mikh80} the operator  $({A}_B + \varepsilon)^{-1} -
(\wt{A}_{F} + \varepsilon)^{-1}$ has the same property and, according to  $\eqref{DM_45}$, the operator
$B-M(-\varepsilon)$ has the same property for all $\varepsilon$ small enough. Using monotonicity of the operator-function  $B -M(x)$ on $(-\infty, 0)$
one obtains $\dim E_{B - M(0)}(-\infty, 0) = n $. This proves (a). The case (b) is proved similarly.
\end{proof}
  \begin{corollary}\label{Cor:DM_10_Posit_Oper}
Let $\{\mathcal{H}, \Gamma_0, \Gamma_1\}$ be a boundary {triplet} for $A^*,$
such that ${A}_0 = \wt{A}_{F}$ and the extensions $\wt{A}_{F}$  and  $\wt{A}_{K}$ are transversal. Let also  $B = B^* \in \mathcal{C}(\mathcal{H})$ and let  $A_B$ be a proper extension of the operator $A$, determined by the equality
 $\dom({A}_B) = \ker(\Gamma_1 - B\Gamma_0)$. Then the following equivalence holds:
$$
     A_B\ge 0 \Longleftrightarrow B -M(0)\ge 0
$$
In particular, the Krein extension  $\wt{A}_{K}$ of the operator ${A}$ corresponds to the operator  $B = M(0)$ via the equality $\dom(\wt{A}_{K})
= \ker(\Gamma_1 - M(0)\Gamma_0)$.
   \end{corollary}
\begin{remark}\label{R:DM_10}
If ${A}_0$ is a positive definite operator, and
$\{\mathcal{H}, \Gamma_0, \Gamma_1\}$ is a positive boundary {triplet}, then Proposition~\ref{R:DM_10} coincides with Theorem~1.6 from~\cite{Mikh80}, which in turn generalizes the results of papers~\cite{Bir56, Kr47}. Indeed, in this case
\[
\ker \Gamma_1 = \dom(A) \dotplus \ker A^* = \dom(A) \dotplus
\mathfrak{N}_0
\] (see \cite{Mikh80}), and thus $M(0) = 0$.
\end{remark}

\section{ Characteristic functions of almost solvable extensions}\label{5}

\subsection{ Characteristic function by A.V. Shtrauss}
Remind, following A.V.~\v{S}trauss \cite{Str60}, the definition of the characteristic function of a proper extension $\wt{A}$ of a Hermitian operator $A$.
\begin{definition}\label{D:DM_8}
Let  ${\cE}$ be a Hilbert space endowed with an inner product $[f,g]_\cE
= (Jf, g)_{\cE}$, where $J=J^*=J^{-1}$ is a signature operator, and let
$\Gamma$ be a linear operator from $\dom(\wt{A})$ to ${\cE}$, such that
$\overline{\ran(\Gamma)} = {\cE}$ and for all  $f, g \in \dom(\wt{A})$
\eq{DM_46}{(\wt{A}f, g) -(f, \wt{A}g) = 2i[\Gamma f, \Gamma g]_\cE.}
The operator $\Gamma$ is called the boundary operator for the extension $\wt{A}$.
\end{definition}

Let $\Gamma '$ be a  boundary operator for $-\wt{A}^*$, acting from $\dom(\wt{A}^*)$ to ${\cE}'$, such that $\overline{\ran(\Gamma ')}
={\cE}'$. For arbitrary  $f \in \dom(\wt{A})$ and $z \in
\rho(\wt{A}^*)$ let us find a vector $g_z \in \dom(\wt{A}^*)$ from the equation
\[
(\wt{A}^* - z)g_z = (\wt{A} - z)f
\]
and define the characteristic function $W(z)$ of  the extension ${A}_B$ by the equality
\[
W(z) \Gamma f = \Gamma ' g_z.
\]
Let $\{ \mathcal{H}, \Gamma_0, \Gamma_1\}$ be a boundary {triplet} for the operator  $A^*$ and let  ${A}_B $ be an almost solvable extension of the operator $A$. In what follows it is assumed that the extensions  ${A}_B$ and
${A}_B^*$ are disjoint, i.e. $\dom({A}_B) \cap
\dom({A}_B^*) = \dom(A)$. The domain of such an  extension ${A}_B$ in the boundary {triplet} $\{ \mathcal{H}, \Gamma_0, \Gamma_1\}$ is defined by
$\dom({A}_B) = \ker(\Gamma_1  - B \Gamma_0)$, where $B \in
[\mathcal{H}]$, $\ker B_I = 0$. Let ${\cE} ={\cE}' = \mathcal{H}$
endowed with the metric $\|f\|_{\cE} = \||B_I|^{1/2}f\|$ and let $J = \sign B_I$. Then the boundary operators  $\Gamma$ and $\Gamma '$ for the extensions ${A}_B$ and
$-{A}_B^*$ can be given by
\begin{equation}\label{eq:BO}
    \Gamma f= \Gamma_0f \quad(f\in\dom({A}_B) ), \quad \Gamma ' g=
\Gamma_0g \quad(g\in\dom({A}_B^*) ).
\end{equation}
 Indeed, in this case one obtains for all $ f, g \in \dom({A}_B) $
\begin{equation}\label{eq:BOid}
    \begin{split}
({A}_B f, g) - (f, {A}_B g) &= (\Gamma_1f, \Gamma_0g)_\cH
-(\Gamma_0f,
\Gamma_1 g)=  2i(B_I \Gamma_0f,\Gamma_0g)_\cH\\
&=2i (|B_I|^{1/2}J\Gamma f, |B_I|^{1/2}\Gamma g )_\cH = 2i [\Gamma
f, \Gamma g]_{\cE}.
\end{split}
\end{equation}
 Similar equality holds also for the operator $-(A_{{B}})^*$.

 \begin{theorem}\label{T:DM_4}
 The characteristic function of almost solvable extension ${A}_B$ is holomorphic on $\rho({A}_{B^*})$, and takes values in $[{\cE}]$ and is given by
 \eq{DM_47}{W(z) = (B^* - M(z))^{-1}(B-M(z)).}
  \end{theorem}
\begin{proof}
For arbitrary $ f \in \dom({A}_B)$ let us find  $g_z \in
\dom({A}_{B^*})$ from the equation
\[
{A}_B f - \wt{A}^*_B g_z = z(f -g_z),
z\in \rho(\wt{A}^*_B).
\]
Then
\[
f - g_z \in \mathfrak{N}_z, \quad \Gamma_1(f - g_z) =
M(z)(\Gamma_0(f-g_z)).
\]
Taking into account that $\Gamma_1 f = B \Gamma_0 f$ and $\Gamma_1 g_z = B^*
\Gamma_0 g_z$, one obtains
\[
B\Gamma_0f-B^*\Gamma_0 g_z = M(z)(\Gamma_0f - \Gamma_0g_z).
\]
Hence
\[
(B - M(z)) \Gamma_0f = (B^* - M(z))\Gamma_0 g_z.
\]
Since $z \in \rho({A}^*_B)$, then by Proposition~\ref{P:DM_7},
$0 \in \rho(B^* - M(z)).$ Therefore,

\eq{DM_48}{\Gamma_0 g_z = (B^* - M(z))^{-1}(B-M(z))\Gamma_0 f.}
The equality $\eqref{DM_47}$ follows from the definition of the characteristic function $W(z)$ and the equality $\eqref{DM_48}$. Analyticity of  $W(z)$ and relation  $W(z)
\in [{\cE}']$ for $ z\in \rho(\wt{A}^*_B)$ follow from the equality
$$W(z) = (B^* - M(z))^{-1}(B-M(z))= I +2i(B^* - M(z))^{-1}B_I.$$
This completes the proof.
\end{proof}

\subsection{Class $\Lambda_J$}
Hereinafter, we will need the following, more general, construction of boundary operators and the characteristic function
$W(z)$. Let us include the operator $B \in [\mathcal{H}]$ in an operator colligation
$\Theta = (B, \mathcal{H}, K, {\cE}, J)$ . Remind (see~\cite{Brod69}), that a set $\Theta = (B, \mathcal{H}, K, {\cE}, J)$, consisting of Hilbert spaces  $\cH$ and $\cE$, and operators $B\in [\cH]$,
$K\in [\cE,\cH]$ and $J\in [\cE]$, is called an operator colligation, if
$J$ is a signature operator in $\cE$, i.e. $J=J^*=J{=1}$ and
\begin{equation}\label{eq:Node}
    B_I = KJK^*.
\end{equation}
If, in addition to~\eqref{eq:Node} the assumption
\begin{equation}\label{eq:Node1}
    \quad\ker B_I = \{0\}.
\end{equation}
is in force, then $\overline{{\ran}(K)} =\mathcal{H}$ and hence the operator
$K^*$ is invertible.
\begin{proposition}\label{prop:ChF2}
   Let $\{ \mathcal{H}, \Gamma_0, \Gamma_1\}$ be a boundary {triplet} for the operator $A^*$, let $M(z)$ be the corresponding Weyl function,
$B\in[\cH]$ and let  ${\cE}=\cE'$ be a Hilbert space endowed with an inner product $[f,g]_\cE = (Jf, g)_{\cE}$, where $J=\sign(J)$. Then boundary operators
 $\Gamma$ и $\Gamma '$ for the extensions ${A}_B$ and $-{A}_B^*$ can be defined by
\begin{equation}\label{eq:RedOp2}
    \Gamma f= K^*\Gamma_0f \quad(f\in\dom({A}_B) ), \quad \Gamma ' g=
K^*\Gamma_0g \quad(g\in\dom({A}_B^*) )
\end{equation}
and the corresponding characteristic function of almost solvable extension  ${A}_B$ takes the form
\begin{equation}\label{DM_49}
    W(z) = I+2iK^*(B^* - M(z))^{-1}KJ\quad (z\in\rho(A_{B^*})).
\end{equation}
\end{proposition}
\begin{proof}
Indeed, plugging~\eqref{eq:Node} into \eqref{eq:BOid}, one obtains
\begin{equation}\label{eq:BOid2}
    \begin{split}
({A}_B f, g) - (f, {A}_B g) &=   2i(B_I \Gamma_0f,\Gamma_0g)_\cH\\
&=2i (JK^*\Gamma_0 f, K^*\Gamma_0 g )_\cH = 2i [\Gamma f, \Gamma
g]_{\cE},
\end{split}
\end{equation}
i.e. $\Gamma$ is a boundary operator for the extension ${A}_B$.
Checking the analogous equality for the operator $-(A_{{B}})^*$,
one convinces that $\Gamma'$ is a boundary operator for the extension
$-{A}_{B^*}$.

Applying the operator $K^*$ to both parts of the equality $\eqref{DM_48}$ one obtains
$$\Gamma ' g_z =
[I+2iK^*(B^* - M(z))^{-1}KJ]\Gamma f.$$
Therefore, the characteristic function of almost solvable extension
 ${A}_B$ takes the form~\eqref{DM_49}. Clearly, $W(z)$ is a holomorphic operator function on
$\rho(\wt{A}^*_B)$ with values in $[{\cE}]$.
\end{proof}

\begin{definition}\label{D:DM_9}
An  operator function $W(z)$ with values in
$[{\cE}]$, holomorphic on a domain $Z_W$, is related to the class
$\Lambda_J$, if it can be represented in the form $\eqref{DM_49}$, where $B =
KJK^* \in [{\cE}]$. Let us say that $W$ belongs to the class $\Lambda^0_J$, if $W\in\Lambda_J$ and, in addition,  $\ker(K^*)= \{0\}$.
\end{definition}

Let us define in $\mathbb{C}_+ \cup \mathbb{C}_-$ the operator function

\eq{DM_50}{V(z) = K^*(B_R - M(z))^{-1}K.} Clearly, for $z \in
\rho({A}_{B^*})\cap (\mathbb{C}_+ \cup \mathbb{C}_-)$ there exists
$(W(z)+I)^{-1} \in [{\cE}]$ and the following equality holds
\eq{DM_51}{V(z) = -i(W(z) - I)(W(z)+I)^{-1}J.} Indeed, multiplying the relation
\[(B^*-M(z))^{-1} - (B_R-M(z))^{-1}= i (B^* -
M(z))^{-1}KJK^*(B_R - M(z))^{-1}\] by $K^*$ from the left, and by  $KJ$ from the right,
yields
\[W(z) - I - 2iV(z)J = i(W(z) - I)V(z)J,
\]
which implies
\[
(W(z)+I)(I - iV(z)J) = 2I.
\]
Similarly, one proves the equality
\[
(I - iV(z)J)(W(z)+I) =
2I.
\]

Let us introduce a rigging
 ${\cE}_+^W \subset {\cE} \subset {\cE}_-^W$
of a Hilbert space ${\cE}$ (см.~\cite{Ber65}), completing ${\cE}$ by the norm

\eq{DM_52}{\|f\|^2_{\cE^W_-} = (\Im V(z_0)f, f)_{\cE}, \quad z_0 \in
\mathbb{C}_+.}
It is easily seen, that the norms obtained via~\eqref{DM_52} are equivalent for different $z_0 \in \mathbb{C}_+.$

\begin{proposition}\label{P:DM_12}
Let $W \in \Lambda^0_J$. Then
\begin{itemize}
\item[\;\;\rm (a)] $V(z) \in [{\cE}^W_-, {\cE}^W_+]$ for all $z \in
\mathbb{C}_+ \cup \mathbb{C}_-;$
\item[\;\;\rm (b)] $(W(z) - I)J$ admits a continuation to a topological isomorphism of spaces ${\cE}^W_-$ and ${\cE}^W_+$;
\item[\;\;\rm (c)] $W(z) + I$ is a topological isomorphism of space ${\cE}^W_+$ into itself.
\end{itemize}
\end{proposition}

\begin{proof}
Let ${\cE}_+$ be a Hilbert space  $ \ran(K^*)$ with the norm
\[
\|v\|_{\cE_+} = \|(K^*)^{-1}v\|\quad(v\in\cE_+).
\]
Then $K^*$
is an isometry in $[\mathcal{H}, {\cE}_+]$.

Consider a rigging ${\cE}_+^W \subset {\cE} \subset {\cE}_-^W$
of the Hilbert space ${\cE}$ (see~\cite{Ber65}). Then for $u\in\cE$ one obtains
\[
\|u\|_{\cE_-}=\sup_{f\in\cE_+,f\ne
0}\frac{|(u,K^*h)_\cE|}{\|h\|_\cH}=\|Ku\|_\cH.
\]
Thus, ${\cE}_-$ is a completion of the space ${\cE}$ in the metric
\[
\|u\|_{\cE_-} = \|Ku\|_\cH\quad(u\in\cE),
\]
and the operator $K$ admits a continuation to an isometrical isomorphism in  $[{\cE}_-, \mathcal H]$ (also denoted by $K$).

It follows from the equality $\eqref{DM_50}$ that
\[
\Im V(z_0) = K^*(B^*_R - M(z_0))^{-1}\Im M(z_0)(B_R - M(z_0)^*)K.
\]
Since $\Im M(z_0)$  is a topological isomorphism in $\mathcal{H}$, then $\Im V(z_0)$
is a topological isomorphism from ${\cE}_-$ onto ${\cE}_+$. Thus, ${\cE}^W_- =
{\cE}_-$ and ${\cE}^W_+ = {\cE}_+$. Now the statement (a) is implied by $\eqref{DM_50}$.

Since
\[
(W(z) - I)J = 2iK^*(B^* - M(z))^{-1}
\] и $(B^* - M(z))^{-1}$ is a topological isomorphism in  $\mathcal{H}$, then $(W(z) - I)J$ is a topological isomorphism of spaces ${\cE}^W_-$ and ${\cE}^W_+$.

It follows from $(\eqref{DM_49})$ that $(W(z) + I){\cE}^W_+ \subset {\cE}^W_+.$
On the other hand, the equality $(W(z)+I)f= g \in {\cE}^W_+$ yields $2f = g+(I-W(z))f \in {\cE}^W_+$.
This proves the statement (c).
\end{proof}

\begin{remark}\label{R:DM_12}
Two characteristic functions corresponding to different pairs of boundary operators differs by constant  $J$-unitary factors $U_1, U_2 \in [{\cE}]$,
\[
W(z) = U_1W_0(z)U_2.
\]
Notice, that the class $\Lambda^0_J$ is not invariant with respect to the multiplication by arbitrary $J$-unitary factor. The set of
$J$--unitary factors, preserving the class  $\Lambda^0_J$, is described by the following theorem.
\end{remark}

\begin{theorem}\label{T:DM_5}
Let $W_0 \in \Lambda^0_J$ and let $U_1, U_2$ be $J$--unitary operators
in ${\cE}$. In order that the operator function  $W(z) = U_1W_0(z)U_2$
to be in the class $\Lambda^0_J$, it is necessary and sufficient that the following hold:
\begin{itemize}
\item[\;\;\rm (a)] The operator $U_2U_1$ is an isomorphism in ${\cE}^{W_0}_+$;
\item[\;\;\rm (b)] $(U_2U_1 - I) \in [{\cE}^{W_0}_-, {\cE}^{W_0}_+]$.
\end{itemize}

\end{theorem}
\begin{proof}
{\it Necessity}. If the operator function $W(z) = U_1W_0(z)U_2$
belongs to $\Lambda^0_J$,  then
\[
W_1(z) = U_1^{-1}W(z)U_1 = W_0(z)U_2U_1
\]
also belongs to $\Lambda^0_J$. By Proposition~\ref{P:DM_12}
\[
(W_0(z)-I)J \in [{\cE}^{W_0}_-, {\cE}^{W_0}_+], \quad (W_1(z)-I)J
\in [{\cE}^{W_1}_-, {\cE}^{W_1}_+] \quad(z\in Z_W).
\]
Let us show that ${\cE}^{W_0}_- = {\cE}^{W_1}_-$. Since \eq{DM_53}{\Im
V_i(z) = (W_i(z)+I)^{-1}(J - W_i(z)JW^*_i(z))(W_i(z)^*+I)^{-1},
\quad(i= 1,2),} then the formula $\eqref{DM_53}$ yields
\[
\|f\|_-^{W_0} = \|(W_1^*+I)(W_0^*+I)^{-1}f\|^{W_1}_-.
\]
Therefore, the operator  $C = (W^*_1 + I)(W_0^* +I)^{-1}$ is an isometry  from
 ${\cE}^{W_0}_-$ onto ${\cE}^{W_1}_-$. By Proposition~\ref{P:DM_12} the identity operator
in ${\cE}$
\[
I = (W^*_1+I)^{-1}C(W^*_0+I),
\]
admits a continuation to an isomorphism from
${\cE}^{W_0}_-$ onto ${\cE}^{W_1}_-$, and thus ${\cE}^{W_0}_- =
{\cE}^{W_1}_-$.

The fact that the $J$ - unitary operator $U = U_2U_1$ is an isomorphism in ${\cE}^{W_0}_+$, follows from the equalities
$$Uf = (I-W_0(z))Uf + (W_1(z) - I)f +f,$$
$$U^{-1}f = (I-W_1(z))U^{-1}f + (W_0(z) - I)f +f,$$
since
$U{\cE}^{W_0}_+ \subset {\cE}^{W_0}_+$ and
$U^{-1}{\cE}^{W_0}_+ \subset {\cE}^{W_0}_+$.

Since $U^{-1}$ is an isomorphism in  ${\cE}^{W_0}_+$, then $(U^{-1})^* =
JUJ$ is an isomorphism in  ${\cE}^{W_0}_-$. By the equality
$$(W_1(z)-I)J = (W_0(z)-I)J(JUJ)+(U-I)J$$
and Proposition  \ref{P:DM_12} one obtains  $(U-I)J \in
[{\cE}^{W_0}_-, {\cE}^{W_0}_+]$.

{\it Sufficiency}. Let ${A}_B$ be an a.s. extension of the operator
$A$, determined by the relation~\eqref{DM_2} in the boundary {triplet} $\{\mathcal{H},
\Gamma_0, \Gamma_1\}$, and let $W_0(z)$ be its characteristic function of the class $\Lambda^0_J$, determined by the formula \eqref{DM_49}. Let us
Let us set $U=U_2U_1$, and let introduce the operator
\[
  X=\left(
      \begin{array}{ll}
        X_{11} & X_{12} \\
        X_{21} & X_{22} \\
      \end{array}
    \right):=
    \left(
                 \begin{array}{ll}
                   X_{11} & -X_{11}B_R+iK(U^*-I)JK^* \\
                   X_{21} & -X_{21}B_R+K^{-*}
                   J(U^*+I)JK^* \\
                 \end{array}
               \right)
\]
\[ = \frac{1}{2} \left(\begin{array}{ll}
K(U^*+I)K^{-1}              & -K(U^*+I)K^{-1}B_R + iK(U^* - I)JK^*\\
iK^{-*}J(I-U^*)K^{-1}   & -iK^{-*}J(I-U^*)K^{-1}B_R +
K^{-*}J(U^*+I)JK^*
 \end{array}
               \right). \]

It is easy to see, that $X_{ij}\in{[{\mathcal H}]}$ $(i,j=0,1)$.
Indeed, since
\[
K^{-1}\in [{\mathcal H},{\cE}_-],\quad
J(I-U^*)=((I-U)J)^*\in[{\cE}_-,{\cE}_+]\quad \mbox{и}\quad
K^{-*}\in[{\cE}_+,{\mathcal H}],
\]
then $X_{21}\in{[{\mathcal H}]}$. Similarly one proves that all other operators
 $X_{ij} ( i,
j\in\{0, 1\})$ are bounded in  ${\mathcal H}$ .
Now the $J$--unitarity of the operator  $X$ is equivalent
to the equalities \eqref{eq:1.34}, \eqref{eq:1.35}. It is easy to see, that
\[
\begin{split}
X_{12}X_{11}^*-X_{11}X_{12}^*&=\frac{i}{4}K\{(U^*-I)J(U+I)-(U^*+I)
J(U-I)\}K^*=0,
\end{split}
\]
since $(U^*+I)J(I-U)=(U^*-I)J(I+U)$. Further,
\[
X_{11}X_{22}^*-X_{12}X_{21}^*=\frac{1}{4}K\{(U^*+I)J(U+I)-(U^*-I)J(I-U)\}JK^{-1}
=I,
\]
\[
X_{21}X_{22}^*-X_{22}X_{21}^*=\frac{i}{4}K^{-*}J\{(I-U^*)J(I+U)+(I+U^*)J(I-U)\}JK^{-1}
=0.
\]
This proves the relation \eqref{eq:1.35}, the relation \eqref{eq:1.34} is checked similarly.

Since the operator $X$ is $J$--unitary, one can define a new boundary {triplet}
$\{{\mathcal H},\wt{\Gamma}_0,\wt{\Gamma}_1\}$ by the formulas
\eqref{DM_18}, \eqref{DM_18A}, setting there $V=I$. Let $\wt{M}(z)$ be the Weyl function of the operator  $A$, corresponding to the boundary {triplet}
$\wt\Pi$ and satisfying the relation
\begin{equation}\label{DM_17A}
\wt M(z)=(X_{11}M(z)+X_{12})(X_{21}M(z)+X_{22})^{-1}.
\end{equation}
Clearly,
\[
X_{11}B+X_{12}=iKU^*JK^*, \quad X_{21}B+X_{22}=K^{-*}JU^*JK^*.
\]
Therefore, $(X_{21}B+X_{22})^{-1}= K^{-*}UK^*\in{[{\mathcal
H}]}$ in view of the assumption а) of the theorem.
By Proposition~\ref{P:DM_6} the domain of the operator $\wt A$ is determined in the new boundary {triplet} $\{{\mathcal H},\wt{\Gamma}_0,\wt{\Gamma}_1\}$ by the equality
$\dom(\wt{A})=\ker(\wt{\Gamma}_1-B\wt{\Gamma}_0)$, where
\[
\wt{B}=(X_{11}B+X_{12})(X_{21}B+X_{22})^{-1}=iKJK^*.
\]
Setting $\wt{K}=(B^*X_{21}^*+X_{22}^*)^{-1}K(U_2^*)^{-1}=KU_1^*$ and including the operator
 $\wt{B}$ into the operator colligation  $\wt{\varphi}=({\mathcal
H},\wt{B};\wt{K},J,\cE)$, one obtains with account of~\eqref{DM_17A}
\[
\begin{split}
\wt{W}(z) &= J\wt{K}^{-1}(\wt{B} - \wt{M}(z))(\wt{B}^*-\wt{M}(z))^{-1}\wt{K}J \\
&= U_1JK^{-1}[\wt{B}(X_{21}M(z)+X_{22}) - (X_{11}M(z)+X_{12})]\times\\
&\times[\wt{B}^*(X_{21}M(z)+X_{22}) - (X_{11}M(z)+X_{12})]^{-1}KJU_1^{-1}\\
&=U_1JK^{-1}[B-M(z)][KU^*K^{-1}(B^*-M(z))]^{-1}KJU_1^{-1} \\
&=U_1JK^{-1}[B-M(z)][B^*-M(z)]^{-1}K(U^*) ^{-1}K^{-1}KJU_1^{-1}  \\
&=U_1W_0(z)UU_1^{-1} = U_1W_0(z)U_2.
\end{split}
\]
This proves the equality $\wt{W}(z)=U_1W_0(z)U_2$.
\end{proof}

\begin{theorem}\label{T:DM_6}
Let ${\cE}$ be a Hilbert space, $J$ be a signature operator in ${\cE}$, let $W(z)$ be an operator valued function with values in  $[{\cE}]$ holomorphic on a domain $Z_W$, which contains a point $z_0 \in\mathbb{C}_+$.
In order that the operator function $W(z)$ to be in the class $\Lambda^0_J$, it is necessary and sufficient that the operator function
 $V(z)$, defined by the equality~\eqref{DM_51},
to admit a holomorphic continuation on $\mathbb{C}_+$ which satisfies the following conditions:
\begin{itemize}
\item[\;\;\rm (a)] $V \in (R)$;
\item[\;\;\rm (b)] $\lim\limits_{y \uparrow \infty}y^{-1}V(iy)=0$;
\item[\;\;\rm (c)]  $\lim\limits_{y \uparrow \infty}y(\Im V(iy)f,f) = \infty \quad (f \in {\cE})$;
\item[\;\;\rm (d)] $V(z) \in [{\cE}^W_-, {\cE}^W_+]  \quad (z \in Z_W)$.
\end{itemize}
\end{theorem}

\begin{proof}
{\it Necessity}. It follows from \eqref{DM_50} that
\[
V(z) = K^*(B_R - M(z))^{-1}K.
\]
Since $(B_R - M(z))^{-1}$ is a  $Q$-function of a densely defined symmetric operator
 $A$, then $V(z)$
satisfies the conditions (a), (b), (c), see Corollary~\ref{C:DM_6}.
The condition (d) was proved in Proposition~\ref{P:DM_11}.

{\it Sufficiency}. Let $\left\langle u,f\right\rangle_{\cE}$ serves for the functional
on $u\in\cE_+$, defined by a vector  $f\in\cE_-$.
Consider the set $\mathcal{L}$ of vector-valued functions $f(z)$ with values in
 ${\cE}_-$, defined on  $\mathbb{C}_+\cup\mathbb{C}_-$
and distinct from zero on a finite set. For arbitrary $z
\in \mathbb{C}_-$, let us set $V(z) = V(\overline{z})^*$ and define an inner product in
$\mathcal{L}$, setting
\eq{DM_54}{(f,g)_V =
\sum_{z,\zeta}\left\langle\frac{V(z) - V^{*}(\zeta)}{z -
\bar\zeta}f(z), g(\zeta)\right\rangle_{\cE}.}
Let
\[
\mathcal{L}_0 = \{f \in \mathcal{L}: (f,g)_V=0 \,\,\mbox{ for all
}\,\,g \in \mathcal{L} \}.
\]
Denote by $H_V$ the completion of the factor-space
$\mathcal{L}/\mathcal{L}_0$ with respect to the metric \eqref{DM_54}.

Define on functions from $\mathcal{L}$, satisfying the assumption
\eq{DM_55}{\chi_1(f) := \sum_{z}f(z) = 0,} an operator $\wt A_1$
by the equality
\[
\wt{A}_1f(z) = zf(z).
\]
It follows from \eqref{DM_54} that the operator $\wt A_1$ is symmetric and in view of
 (b)  the domain  $\dom(\wt {A}_1)$ of $\wt A_1$ is dense in
$H_V$. The closure of the operator $\wt {A}_1$ в $H_V$ will be also denoted by $\wt {A}_1$.

A function from $\mathcal{L}$, which is distinct from zero in a unique point
$z_0$ and is equal to $h$ at $z_0$, will be denoted by $h_{z_0}(z)$:
\[
h_{z_0}(z)=\left\{\begin{array}{cc}
            h ,& \mbox{if }z=z_0; \\
                    0 ,& \mbox{if }z\ne z_0.\\
                  \end{array}\right.\quad(h\in \cE_-).
\]
Let $f \in \mathcal{L}$, $f(z_0)= 0$ and let
$\chi_1(\frac{f(z)}{z-z_0}) = h$. Then
\[
(\wt {A}_1 -z_0)\left(\frac{f(z)}{z-z_0}-h_{z_0}(z)\right)=f(z).
\]
Therefore, $\ran(\wt {A}_1-z_0)=H_V$ and thus $\wt {A}_1$ is a self-adjoint operator in $H_V$.

Notice that $\ker \Im V(z_0) = 0$ for $z_0 \in \mathbb{C}_+$, since the assumption $(\Im V(z_0)h,h)=0$ leads to the condition $(\Im V(z)h,h)=0$ for all
$z \in \mathbb{C}_+$, which contradicts the condition (c).

Consider a linear manifold $\mathfrak{N}_{z_0}:=\{h_{z_0}(\cdot) : h \in
{\cE}\}$ in $H_V$ and define an operator  $\gamma_{z_0}$ by the equality
$\gamma_{z_0}h = h_{z_0}(z).$ It follows from the equality
\eq{DM_56}{\|h_{z_0}(\cdot)\|_V^2 = \left\langle\frac{\Im
V(z_0)}{\Im z_0}h,h\right\rangle_{\cE} = \frac{1}{\Im z_0}\|h\|^2_-}
that the operator  $\gamma_{z_0}$ belongs to $ [{\cE},
\mathfrak{N}_{z_0}]$ and defines a topological isomorphism from
${\cE}_-$ onto $\mathfrak{N}_{z_0}$.

Let $A$ be a restriction of the operator  $\wt {A}_1$ to the linear manifold
\[
D = \{f \in \dom(\wt {A}_1): (\wt {A}_1 - \overline{z}_0)f \perp
\mathfrak{N}_{z_0}\}.
\]
The domain of the operator  $A$ can be characterized by two conditions: $ \chi_1(f) = 0, \chi_0(f) = 0 $, where
\eq{DM_57}{\chi_0(f) = \sum_{z}V(z)f(z). }

Indeed, if $f \in \mathcal{L}$ и $\chi_0(f)=\chi_1(f)= 0$,
then
\[
((\wt {A}_1 - \overline{z}_0)f, h_{z_0}(z))_V = \sum_{z}
\left\langle\frac{V(z) - V(z_0)}{z -\overline{z}_0}(z -
\overline{z}_0)f(z),h\right\rangle_{\cE} = \sum_{z}\left\langle
V(z)f(z),h\right\rangle_{\cE} = 0.
\]

By the assumption (b)  $\mathfrak{N}_{z_0} \cap \dom({A}_1) =
\{0\}$ (see \cite{KL73}), and therefore $A$ is a densely defined symmetric operator. The constructions of the space  $H_V$ and the operators $A$, $\wt{A}_1$ belong to M.G. Kre\u{\i}n and H. Langer \cite{KL73}.

Let $\wt{A}_0$ be the closure in $H_V$ of the operator $\wt{A}_0f(z) =
zf(z)$, defined originally on vector-functions from $\mathcal{L}$,
satisfying the condition $\chi_0(f) = 0.$ Clearly, $\wt{A}_0$ is a symmetric operator. Let us show, that the extensions $\wt{A}_1$ and $\wt{A}_0$ are transversal.
In view of the equality $
\dom(\wt{A}_1) \dotplus \mathfrak{N}_{z_0} = \dom(A^*)$, in order to prove the transversality of $\wt{A}_1$ and $\wt{A}_0$ it is enough to check that
\eq{DM_58A}{
\mathfrak{N}_{z_0} \subset \dom(\wt{A}_1) + \dom(\wt{A}_0).
}
Let $h \in {\cE}_-$. Then
\eq{DM_58}{h_{z_0}(z) = [f_{z_0}(z) -
f_{\bar z_0}(z)]+[h_{z_0}(z) - f_{z_0}(z)+f_{\overline{z}_0}(z)],}
where $f = [2i\Im V(z_0)]^{-1}V(z_0)h \in {\cE}_-.$ The first term in
$\eqref{DM_58}$ belongs to $\dom(\wt{A}_1)$, and the second term belongs to $\dom(\wt{A}_0)$. This proves~\eqref{DM_58A}.

The transversality of the extensions $\wt{A}_0$ and $\wt{A}_1$ implies, in particular, that $\wt{A}_0$ is a self-adjoint operator.

Define a boundary {triplet}  $\{\mathcal{H}, \Gamma_0, \Gamma_1\}$ for $A^*$,
setting
\[
\mathcal{H}= {\cE}, \quad \Gamma_0f=-R^{1/2}\chi_0(f),\quad
\Gamma_1f = R^{1/2}\chi_1(f),
\]
where $R = \Im V(z_0)$. Let us show that the mapping
$f \to \{\Gamma_0f, \Gamma_1f\}$ from $\dom(A^*)$ to $\mathcal{H}
\oplus \mathcal{H}$ is surjective. To prove this it is enough to check
that the operators $\Gamma_0|_{\mathfrak{N}_{Z_0}},
\Gamma_1|_{\mathfrak{N}_{Z_0}}$ are isomorphisms from
$\mathfrak{N}_{z_0}$ to $\mathcal{H}$.

In order to prove the first statement represent  the operator
 $V(z_0)$
in the form  $V(z_0) = Q+iR$, where $Q=Q^*, R = R^* \in [{\cE}_-,{\cE}_+]$.
Then
\[
V^*(z_0)R^{-1}V(z_0) = QR^{-1}Q+R.
\]
Hence one obtains
$$\|\Gamma_0h_{z_0}\|^2_{\cE} = \left\langle R^{-1}V(z_0)h, V(z_0)h\right\rangle_{\cE}
\ge \left\langle Rh,h\right\rangle_{\cE} = \|h\|^2_-=\Im
z_0\|h_{z_0}\|_V^2.$$
On the other hand, it is clear that there is $C>0$, such that
$$\|\Gamma_0h_{z_0}\|^2_{\cE} \le C\|h\|_-^2 = C \Im
z_0\|h_{z_0}\|_V^2.$$
These two inequalities show, that
$\Gamma_0|_{\mathfrak{N}_{z_0}}$ is an isomorphism in ${\cE}$.

The second statement is implied by the equality
$$\|\Gamma_1h_{z_0}\|^2_{\cE} =\left\langle Rh,h\right\rangle_{\cE}
= \|h\|^2_- = \Im z_0\|h_{z_0}\|_V^2.$$
And finally, it follows from the equality
$$(A^*f,g)_V - (f,A^*g)_V = \sum_{z,\zeta}\left\langle[V(z) - V^*(\zeta)]f(z), g(\zeta)\right\rangle_{\cE} = $$
$$= \sum_{z, \zeta}\{\left\langle V(z)f(z),g(\zeta)\right\rangle_{\cE}
- \left\langle f(z), V(\zeta)g(\zeta)\right\rangle_{\cE}=$$
$$=\sum_{z, \zeta} \{(R^{-1/2}V(z)f(z),R^{1/2}g(\zeta))_{\cE} -
(R^{1/2}f(z), R^{-1/2}V(\zeta)g(\zeta))_{\cE}\}=$$
$$=(\Gamma_1f,\Gamma_0g)_{\cE} - (\Gamma_0f, \Gamma_1g)_{\cE}$$
that $\{\mathcal{H}, \Gamma_0, \Gamma_1\}$ is a boundary {triplet} for
 $A^*$. The Weyl function for this boundary {triplet}
takes the form
\[
M(z) =- R^{1/2}V^{-1}(z)R^{1/2}.
\]
Let us set $B = iR^{1/2}JR^{1/2}$ and consider an a.s. extension ${A}_B$
of the operator $A$, determined by the "boundary condition"
\[
\sum_{z}(I+iJV(z))f(z)=0,
\]
which is equivalent to the condition
\[
\Gamma_1f=iR^{1/2}JR^{1/2}\Gamma_0f.
\]
Setting $K=R^{1/2}J$, one can find the characteristic function
$\wt{W}(z)$ of the operator ${A}_B$ by the formula~\eqref{DM_49}
\eq{DM_59}{\wt{W}(z) = I
+2iJR^{1/2}(B^*-M(z))^{-1}R^{1/2}=I+2iJV(z)(I - iJV(z))^{-1}=}
$$=(I+iJV(z))(I-iJV(z))^{-1}=W(z).$$
This proves that $W\in\Lambda_J^0$.
\end{proof}

\begin{remark}\label{R:DM_13}
Characteristic functions of unbounded operators with finite non-Hermitian rank have been studied by the methods  of the theory of bi-extensions of symmetric operators in papers by E.R. Tsekanovskii and Yu.L. Shmuljan (see~\cite{Ts65,ShmTs77}), and in the case of infinite non-Hermitian rank in papers by Yu.M. Arlinskii and E.R. Tsekanovskii (see~\cite{Arl76,ArlTs74}).
In particular, in~\cite{Arl76,ArlD79}
the class of characteristic functions of nonbounded operator colligations constructed within the bi-extension theory was completely characterized. Methods of boundary {triplet}s allow to present an alternative approach to the theory of characteristic functions of nonbounded operators, which nonetheless leads to the same class $\Lambda_J$. Statements, which are closed to Proposition~\ref{P:DM_12} and Theorem~\ref{T:DM_5}, were proved earlier in~\cite{ArlTs74}.
\end{remark}
\section{Ordinary Differential Operators with Bounded Operator Coefficients}\label{6}
\subsection{Operators on finite intervals} Let $H$ be a separable Hilbert space, and let $A$ be a minimal operator
generated in $L_2([0,b],H)$ by the differential expression of order $2n$ of the form
  \eq{DM_60}
{l[y] = \sum_{k=1} ^{ n} (-1)^k (p_{n-k}y^{(k)})^{(k)}+p_n y.}
Here $y: [0,b]\to H$ is a vector function with values in $H$, quasi-derivatives
$y^{[k]}$ are successively defined by
\[
y^{[k]}(x)=y^{(k)}(x)\,\, (k=0,\dots,n-1), \quad
y^{[n]}(x)=p_0(x)y^{(n)}(x),
\]
\[
y^{[n+k]}(x)=p_k(x)y^{(n-k)}(x)-\frac{d}{dx}y^{[n+k-1]}(x)\,\ \  (k=1,\dots,n),
\]
and the coefficients $p_k=p_k(t)$ satisfy the following conditions:
\[
p_k(t) = p_k(t)^* \in [H],\quad p_0(t)^{-1} \in [H] \quad\mbox{ for
}\quad 0 \le t \le b, \quad p_k \in C^{n-k}([0,b],[H])\quad(k =
0,1,...,n).
\]
The domain of the minimal operator $A$ is of the form
\[
\dom A=\{y\in\dom A^*:\, y^{[k]}(0)=y^{[k]}(b)=0, \  k=0,1,\dots,2n-1\}.
\]
Due to Rofe-Beketov\cite{RB69}, a boundary triplet for $A^*$ has the form
$\{\mathcal{H}, \Gamma_0, \Gamma_1\}$, where $\mathcal{H} = H^{2n},$
\[
\begin{split}
\Gamma_0y &= (y(0),...,y^{(n-1)}(0); y(b),...,y^{(n-1)}(b))^T,\\
\Gamma_1y &= (y(^{[2n-1]}0),...,y^{[n]}(0);
-y^{[2n-1]}(b),...,-y^{[n]}(b))^T.
\end{split}
\]

Let $Y_i(z,t)$ be a solution to the operator equation $l[Y] = zY$
satisfying the initial conditions
\[
Y_i^{[j-1]}(z,0) = \delta_{ij} I,  \quad i,j = 1, ..., 2n.
\]
From the results of~\cite[Theorem 6]{RB69} it follows that the block operator
\[Y(z,t) =\|Y_1(z,t), ..., Y_{2n}(z,t)\|\]
establishes an isomorphism between $H^{2n}$ and
$\mathfrak{N}_z(A)$. By definition of the Weyl function, 
the following equality  holds  $ M(z)\Gamma_0y = \Gamma_1y $, $y\in\sN_z$. By substituting for $y$
columns of the matrix $Y(z,t)$ we arrive at the relation
\[
M(z)Y^0(z) = Y^1(z),
\]
in which operator functions $Y^i(z)=\Gamma_iY(z,t)$, $(i =
0,1)$ are isomorphisms in $\mathcal{H}$  for all $z \ne
\overline{z}$. Consequently, \eq{DM_61}{M(z) =
Y^1(z)[Y^0(z)]^{-1}.} Further, let $B\in[\cH]$, and let ${A}_B$ be the almost
solvable extension of the operator
$A$ defined by the boundary condition $\Gamma_1y =B\Gamma_0y$. Consider the operators
\[
\Phi_z = (\Gamma_1 - B\Gamma_0)\mid_{\mathfrak{N}_z},\quad \Phi_{*z}
= (\Gamma_1 - B^*\Gamma_0)\mid_{\mathfrak{N}_z}.
\]
Since the operator $\Phi_{*z}$ is an isomorphism from $\mathfrak{N}_z$ onto
$\mathcal{H} $ for any $ z \in \rho(\wt{A}^*_B)$, then it follows from the
formula~\eqref{DM_47}, that the characteristic function for ${A}_B$
is of the form \eq{DM_62}{W(z) =[\Phi_{*z}Y]^{-1}[\Phi_zY].}

\begin{remark}\label{R:DM_14}
For the operator $\wt{A} \supset A$ defined by expression $\eqref{DM_60}$
and by the conditions
\[
y^{[i]}(0)=0\quad (0 \le i \le 2n-1),
\]
and for some choice of a boundary triplet one obtains:
\[
\Phi y = \{y(0), ..., y^{[2n-1]}(0)\},\quad\Phi_* y = \{y(b), ...,
y^{[2n-1]}(b)\}.
\]
Hence $W(z)$ takes the form
\[
W(z)^{-1} = \|Y_j^{i-1}(b)\|_{ij=1}^{n}
\]
(see \cite{Str60}). As in Remark~\ref{R:DM_3}, the equality $Y(z, t) = (\Phi
\mid_{\mathfrak{N}_z})^{-1}$ yields Theorem~6 from~\cite{RB69}.
\end{remark}

\subsection{Operators on the half-line}
Let $A$ be a minimal operator generated in $L_2(0,\infty)$ by the differential
expression~\eqref{DM_60} in which $p_0^{-1},p_1,\dots,p_n$ are real and measurable functions
on $(0,\infty)$ integrable on each subinterval $[a,b]$ of $(0,\infty)$.
Suppose also that the operator $A$ has deficiency indices $n_+(A) = n_-(A) = n$.
Then the domain of the minimal operator $A$ is of the form
\[
\dom A=\{y\in\dom A^*:\, y^{[k]}(0)=0 \mbox{ for any  }
k=0,1,\dots,2n-1\}.
\]
For an arbitrary boundary triplet $\{\mathcal{H}, \Gamma_0,\Gamma_1\}$,
choose the basis $y_k(z,x) (1 \le k \le n)$ in $\mathfrak{N}_z$ such that
\[
I_n = \Gamma_0Y=\Gamma_0( y_1, y_2, ..., y_n).
\]
Then the Weyl function $M(z)$ of the operator $A$ is given by
\[
M(z) = \Gamma_1Y = \Gamma_1( y_1, y_2, ..., y_n).
\]
If $g_k(\cdot,z) \in \mathfrak{N}_z$ and $g_k^{[n-i]}(0,z) =
\delta_{ik}$, then, by defining $\Gamma_0$ and $\Gamma_1$ by the equalities
\eq{DM_63}{\Gamma_0y =\{y^{(n-1)}(0), ... , y(0)\}^T, \qquad
\Gamma_1y =\{y^{(n)}(0), ... , y^{(2n-1)}(0)\}^T} and by setting
$y_k(x,z) = g_k(x,z)$, one obtains
\[
M(z) = \|M_{jk}(z)\|_{k=1}^n,\quad\mbox{ where }\quad M_{jk}(z) =
g_k^{[n+j-1]}(0,z).
\]
Thus, in this case the Weyl function $M(z)$ exactly coincides with the
characteristic matrix of the operator ${A}_0$ (see~\cite[p. 278]{Naj69}).
\footnote{Let $G(x,\xi,z)=\sum_{k=1}^n u_k(x,z)g_k(\xi,z)$ $(x<\xi)$ be the Green function of the Dirichlet problem, i.e. the kernel of the resolvent of the operator $A_0$. Then the above-mentioned coincidence   becomes evident, if one uses the following identities for quasi-derivatives $G_{\nu, j}$ $(\nu+j=2n-1)$ of the Green function\\
$
G_{\nu, j}(\xi-0,\xi,z)-G_{\nu, j}(\xi+0,\xi,z)=
\sum_{k=1}^n \left\{u_k^{[\nu]}(\xi,z)g_k^{[j]}(\xi,z)-
u_k^{[j]}(\xi,z)g_k^{[\nu]}(\xi,z)\right\}=\left\{\begin{array}{cc}
                                                    1 ,& j<\nu \\
                                                    -1, & j>\nu,
                                                  \end{array}\right.
$\\
which imply that for $u_k^{[n+j-1]}(a)=\delta_{k,j}$ one has $g_k^{[n-i]}(a)=\delta_{k,j}$. Mention also, that the book \cite[p.260]{Naj69} contains an unfortunate inaccuracy: a wrong sign $(-1)^{n-1}$ in the right part of the above equality}
Suppose that the almost solvable extension $\wt{A}$ of the operator $A$ is defined by
\eq{DM_64}{\Phi_j y = \sum_{1 \le k \le 2n}\phi_{jk}y^{[k-1]}(0)=0,
\quad 1 \le j \le n,} and the extension $(\wt{A})^*$ is defined by
\eq{DM_65}{\Phi_{*j} y = \sum_{1 \le k \le
2n}\phi_{*jk}y^{[k-1]}(0)=0, \quad 1 \le j \le n.} Since $\wt{A}$ is
an almost solvable extension of the operator $A$, it follows that, at some boundary triplet
$\{\mathcal{H}, \Gamma_0, \Gamma_1\}$, one has:
\[
\wt{A} = {A}_B,\quad \dom(\wt{A}) = \ker(\Gamma_1 - B\Gamma_0),\quad
B \in [\mathcal{H}],
\]
and conditions~\eqref{DM_64}, \eqref{DM_65} take the form
\[
\Phi y = C_1(\Gamma_1 -B\Gamma_0)y = 0,\quad \Phi_* y =
C_2(\Gamma_1- B^*\Gamma_0)y = 0 ,
\]
where $C_1, C_2, C_1^{-1}, C_2^{-1} \in [\mathcal{H}]$. By introducing the matrices
\[
\Phi(z) = \|\Phi_jy_k(x,z)\|^n_{j,k=1}, \quad
 \Phi_*(z) = \|\Phi_{*j}y_k(x,z)\|^n_{j,k=1}
 \]
and by taking into account~\eqref{DM_47} we arrive at the following
expression for the characteristic function $W(z)$ of the operator ${A}_B$:
 \eq{DM_66}{W(z) = [(\Gamma_1-B\Gamma_0)Y][(\Gamma_1 - B^*\Gamma_0)Y]^{-1} = C^{-1}_1\Phi(z)\Phi_*(z)C_2.}

 \begin{remark}\label{R:DM_15}
At the "natural"\ boundary triplet of form~\eqref{DM_63}, the matrices $C_1$ and
$C_2$ have the form
\[
C_1 = \|\phi_{j,n+k}\|^n_{jk=1}, \quad C_2 =
  \|\phi_{*j,n+k}\|^n_{jk=1}.
\]
All the conclusions of this section are valid without change for the
differential operation of form~$\eqref{DM_60}$ in $L_2([0, \infty), H)$ only if
$n_+(A)=n_-(A) = n =\dim H < \infty$. Note that the computation of the characteristic function of the differential operator in the scalar case ($\dim H = 1$) is held in~\cite{Arl80,Kuz64}
in which similar results are obtained by other techniques.
 \end{remark}

 \section{Sturm--Liouville Operator with Semi-Bounded Operator Potential} \label{7}

8.1. Let $\mathfrak{H}$ be a separable Hilbert space, and let $L= L_{\min}$ be a minimal operator
generated in $L_2([0,b],\mathfrak{H})$ by the differential expression
\[
l[y] = - y''+Ay+q(t)y
\]
with an unbounded operator $A=A^*\ge I$,\, $A \in \mathcal{C}(\mathfrak{H})$.
Let also $\{\sH_\alpha\}$ $(\alpha\in\dR)$ be the scale of Hilbert spaces
constructed with respect to the operator $A$:
\[
\sH_\alpha=\dom(A^\alpha),\quad \|h\|_{\sH_\alpha}=\|A^\alpha
h\|_\sH \quad (h\in \sH_\alpha).
\]

In~\cite{Gor71}, M. L. Gorbachuk showed that for any
$y \in \dom(L^*)$ boundary values $y(0)$ and $y(b)$  exist in the space
$\mathfrak{H}_{- 1/4}$ and indicated the following boundary triplets for the operator
$L^*$:
\[
\mathcal H = \mathfrak{H} \oplus \mathfrak{H}, \quad \Gamma_0 y =
\{y_0, -y_b\}^T,\quad \Gamma_1y = \{y_0', y_b'\}^T,
\]
where
 \begin{equation}\label{DM_Boun_Map}
y'_0=A^{1/4}[y'(0)+A^{1/2}y(0)],\quad  y_0 = A^{-1/4}y(0), \quad y(0) \in
\mathfrak{H}_{-1/4}.
   \end{equation}
Further, in the case $q(t) \equiv 0$ there holds the following
representation  for any $y \in \mathfrak{N}_z$ (see \cite{Gor71}):
  \eq{DM_67}
{y(t,z) = \omega_1(t,z)f_1+\omega_2(t,z)f_2, \quad (f_1,f_2 \in \mathfrak{H}),}
  \eq{DM_68}
  {\omega_1(t,z)=e^{-t\sqrt{A-z}}A^{1/4}, \quad
\omega_2(t,z) =\frac{\shh
t\sqrt{A-z}}{\sqrt{A-z}}A^{3/4}e^{-b\sqrt{A}.}}
Therefore, the Weyl function $M(z)$ is of the form
\[
M(z) = \Omega_1(z)\Omega_0^{-1}(z),
\]
where
\[
\Omega_0(z) = (\Gamma_0\omega_1(t,z), \Gamma_0\omega_2(t,z)), \quad \Omega_1(z)=
(\Gamma_1\omega_1(t,z), \Gamma_1\omega_2(t,z)).
\]

8.2. Let $B \in \mathcal{C}(\mathcal{H})$, and let ${L}_B$ be the extension
of the minimal operator $L$. Then, due to Proposition~\ref{P:DM_7}, this yields a full
characterization of the spectrum of the operator ${A}_B$. For example,
\[
z \in \sigma_p({L}_B) \Longleftrightarrow 0 \in \sigma_p(M(z)-B).
\]
      \begin{proposition}\label{P:DM_13}
Suppose that $A^{-1} \in \sS_{p-1/2}$, $z \in \rho({L}_B)$, $\xi \in \rho(B)$.
Then the following equivalence  holds
\[
({L}_B-z)^{-1} \in \sS_p(H) \Longleftrightarrow (B - \xi)^{-1} \in
\sS_p(\mathcal{H}).
\]
 \end{proposition}
The proof is clearly implied by both Theorem $\ref{DM_2}$ and the equivalence
\[
A^{-1} \in \sS_{p-1/2} \Longleftrightarrow \wt L_D^{-1} \subset
\sS_p
\]
established in~\cite{GG72}, where $\wt L_D=L_0$
is the Dirichlet extension of the operator $L$.

As is shown in~\cite{GG72}, vector functions $y \in \dom(L^*)$ are continuous on
$(0,b)$ in the space $\mathfrak{H}_{3/4}$, and continuous on $[0,b]$
only in $\mathfrak{H}_{-1/4}$. However, vector functions from
$\dom({L}_0)$ preserve continuity in $\mathfrak{H}_{3/4}$ and at the endpoints
of the segment.

\begin{definition}\label{D:DM_10}{\rm(\cite{GG72})}.
An extension ${L}_B$ of the operator $L$ is said to be $\alpha$-smooth if
$\dom({L}_B) \subset C([0,b], \mathfrak{H}_{\alpha})$ whenever $-1/4<\alpha<3/4$,
and maximally smooth whenever $\alpha = 3/4$.
\end{definition}

\begin{proposition}\label{P:DM_14}
For an extension ${L}_B$, $(B \in \mathcal{C}(\mathcal
H))$ to be $\alpha$-smooth, it is necessary and sufficient that
\[
(A \oplus A)^{\alpha+1/4}(B-\xi)^{-1} \in [\mathcal{H}] \quad \mbox{
for }\quad \xi \in \rho(B).
\]
\end{proposition}

\begin{corollary}\label{C:DM_12}
If $A^{-1} \in \sS_{\infty}(\mathfrak{H})$, then any $\alpha$-smooth
extension has a discrete spectrum.
\end{corollary}

\begin{proposition}\label{P:DM_15}
Let $B \in \cC(\mathcal{H})$, and let ${L}_B \supset L$. Then ${L}_B$ is
$1/4$-smooth if and only if ${L}_B$ is the extension with the finite Dirichlet
integral, i.e.,
\eq{DM_69}
{D[y] = \int_{0}^{b}[\|y'(t)\|^2+\|A^{1/2}y(t)\|^2+(q(t)y(t),y(t))]dt<
\infty.}
\end{proposition}
\begin{definition}\label{D:DM_11}
Let $B  \in \cC(\mathcal{H})$. An extension ${L}_B$ is said to be
${D}$-extension if $$({L}_By,y)={D}[y] \quad \mbox{ for any }\quad  y \in \dom({L}_B).$$
\end{definition}
      \begin{proposition}\label{P:DM_16}
Suppose that $B  \in \cC(\mathcal{H})$ and  $\wh{A} = A \oplus A$.
For the extension ${L}_B$ to be the ${D}$-extension, it is necessary and sufficient that
${L}_B$ be $1/4$-smooth and
\[
(Bg,g)=\|\wh A^{1/2}g\|^2 \quad
 \mbox{ for any }\quad g \in \dom(B).
 \]
        \end{proposition}
\begin{proof}
If ${L}_B$ is the ${D}$-extension, then Proposition~\ref{P:DM_15} implies that ${L}_B$
is $1/4$-smooth and $\dom(B) \subset \dom(\wh A^{1/2})$. For $y\in \dom({L}_B)$ we put
\[
Y=\left(%
\begin{array}{c}
  y(0) \\
  -y(b) \\
\end{array}%
\right),\quad
Y'=\left(%
\begin{array}{c}
  y'(0) \\
  y'(b) \\
\end{array}%
\right).
\]
Then
$$Y=\wh{A}^{1/4}\Gamma_0y \in H_{1/4},\quad
Y' = [-\wh{A}^{3/4}+\wh{A}^{-1/4}B]\Gamma_0y \in H_{-1/4}.$$
By integrating by parts one represents the expression $({L}_By,y)$ as
\[
({L}_By,y) = (Y', Y)+D(y).
\]
(see ~\cite{Gor71}). This yields the equality $(Y',Y)=0$ that, in view of
the above representations for $Y$ and $Y'$, takes the form
\[
(B\Gamma_0y, \Gamma_0y)=(\wh{A}\Gamma_0y, \Gamma_0y).
\]
By putting $g=\Gamma_0y$ and by taking into account that $g \in \dom(B)\subset
\dom(\wt{A}^{1/2})$ we obtain $(Bg,g) = (\wt{A}^{1/2}g,
\wt{A}^{1/2}g)$. The above reasoning is convertible. Thus, proposition is proved.
\end{proof}
  \begin{corollary}\label{C:DM_13}
Suppose that $B  \in \cC(\mathcal{H})$ and the operator ${L}_B$ is the ${D}$-extension. Then
${L}_B$ is symmetric.
\end{corollary}
Propositions~\ref{P:DM_14}, \ref{P:DM_16} and Corollaries~\ref{C:DM_12}, \ref{C:DM_13}
generalize the results of~\cite{GG72,GG72A} on smoothness of
dissipative extensions to the case of almost solvable ones.

Let $\wt{L}_1, \wt{L}_2$ be the extensions of the operator $L$
generated by the conditions $Y'=S_jY$, $j =1,2$. Here $S_j$ are strictly
$\wh{A}^{1/2}$ - bounded operators\footnote[1]{I. e., $\dom(\wh{A}^{1/2})\subset \dom(S_j)$ and $\|S_jf\| \le a\|\wh{A}^{1/2}f\|+b\|f\|$, $a<1$, for any $f \in \dom(\wh{A}^{1/2})$.} in
$\mathcal{H}$. Then (see~\cite{GG72}) we conclude that
   \begin{equation}\label{8.4}
\wt{L}_j={L}_{B_j},\quad\mbox{ where  }\quad
B_j=\wh{A}^{1/4}(\wh{A}^{1/2}+S_j)\wh{A}^{1/4}, \ \ \ j = 1,2.
  \end{equation}
    \begin{proposition}\label{P:DM_17}
For the $\sS_p$ -- resolvent comparability of the operators ${L}_{B_1}$ and
${L}_{B_2}$ in $L_2([0,b], \mathcal{H})$ it is sufficient that
$(S_1-S_2)\wh{A}^{-1/2} \in \sS_{p_1}$ and $\wh{A}^{-1} \in \sS_{p_2}$,
where $p=p_1p_2(p_1+p_2)^{-1} \ge 1$.
       \end{proposition}
\begin{proof}
Due to Theorem~\ref{T:DM_2}, the  $\sS_p$ -- resolvent comparability of the operators
$\wt{L}_{B_1}, \wt{L}_{B_2}$ in $L_2([0,b], \mathcal{H})$
is equivalent to one of the operators
$B_1$ and $B_2$. Let $z \in \rho(B_1)\cap\rho(B_2)$. Without loss of generality
we may assume that $\|S_j\wh{A}^{-1/2} - z\wh{A}^{-1}\|<1$, otherwise
it suffices to take the operator $A+\eta I$, $\eta > 0$, instead of $A$
in~$\eqref{DM_67}$. Then the operators $(I + S_j\wh{A}^{-1/2} - z\wh{A}^{-1})$
have bounded inverse, and the following equality holds:
\eq{DM_70}
{(B_1-z)^{-1} - (B_2-z)^{-1}=}
$$
= \wh{A}^{-3/4}(I+S_1\wh{A}^{-1/2}-z
\wh{A}^{-1})^{-1}(S_2-S_1)\wh{A}^{-1/2}(I+S_2\wh{A}^{-1/2}-z\wh{A}^{-1})\wh{A}^{-1/4}.
$$
Since $(S_1-S_2)\wh{A}^{-1/2} \in \sS_{p_1}$ and $\wh{A}^{-1} \in \sS_{p_2}$, we have
$(B_1-z)^{-1} - (B_2-z)^{-1}\in \frak S_p.$
\end{proof}

\begin{corollary}\label{C:DM_14}
If  $S_j$ and $S_j^*$ are strictly $\wh{A}^{1/2}$-bounded operators in $\mathcal{H},\ j
=1, 2$, then for the $\sS_p$ -- resolvent comparability of the extensions $L_1$ and $L_2$
it is sufficient to satisfy any of the following conditions:
\begin{itemize}
\item[\;\;\rm (a)] $\wh{A}^{-1/2}(S_1-S_2)\wh{A}^{-1/2} \in \sS_p$ ;
\item[\;\;\rm (b)] $\wh{A}^{-1/2}(S_1-S_2)\wh{A}^{-1/2} \in \sS_{p_1}, \quad
\wh{A}^{-1/2} \in \sS_{p_2}, \quad p_1^{-1}+p_2^{-1}=p^{-1} \le 1$.
\end{itemize}
\end{corollary}

\begin{proof}
Indeed, in this case the operators $(I+\wh{A}^{-1/2}S_j -
z\wh{A}^{-1})^{-1}$ are bounded, and equality $\eqref{DM_70}$
can be given in the symmetric form
$$(B_1-z)^{-1} - (B_2-z)^{-1}=\\$$
$$=\wh{A}^{-1/4}(I+S_1\wh{A}^{-1/2}-z \wh{A}^{-1})^{-1}\wh{A}^{-1/2}(S_2-S_1) \wh{A}^{-1/2}(I+S_2\wh{A}^{-1/2}-z\wh{A}^{-1})^{-1}\wh{A}^{-1/4}.$$
\end{proof}

\begin{remark}\label{R:DM_16}
Both Proposition~\ref{P:DM_17} and Corollary~\ref{C:DM_14}
in the self-adjoint case $\wt{L}_j =\wt{L}_j^*$ were proved in~\cite{GK82}.
\end{remark}

\begin{example}\label{Ex:DM_1}
Let $L= L_{\min}$ be minimal operator generated in $L_2([0, \pi] \times
(-\infty, +\infty))$ by the Laplace expression
\[
l[y](t,x) = -\Delta y = -(\partial^2/\partial t^2 +
\partial^2/\partial x^2)y.
\]
Let also $\wt{L}_{\sigma_j}$ $(j=1,2)$ be extensions generated by the conditions
$$
[\partial y/\partial t-\sigma_{j,0}(x)y]\mid_{t=0} =0, \quad [\partial y/\partial
t-\sigma_{j,\pi}(x)y]\mid_{t=0}=0,
$$
where $\sigma_{j,0}, \sigma_{j,\pi} \in L_{\infty}(-\infty, +\infty),\ j = 1,2.$

Writing the operator $L_0^*$ in the form $L_0 = -y''+Ay-y$, where $A =
-d^2/dx^2 + I \ge I$ in $L_2(-\infty, +\infty)$, one can apply the above
assertions to it. The Weyl function $M(\lambda)$ is of the form
$$
M(\lambda) = \wh{A}^{1/4}\begin{pmatrix}
\sqrt{A}-{\displaystyle\frac{\sqrt{A-I-\lambda}}{\tan \pi
\sqrt{A-I-\lambda}}}
\quad -{\displaystyle\frac{\sqrt{A-I-\lambda}}{\shh\pi \sqrt{A-I-\lambda}}}\\
-{\displaystyle\frac{\sqrt{A-I-\lambda}}{\shh\pi
\sqrt{A-I-\lambda}}} \quad \sqrt{A}-
{\displaystyle\frac{\sqrt{A-I-\lambda}}{\tan \pi \sqrt{A-I-\lambda}}}
\end{pmatrix} \wh{A}^{1/4}.
$$

Both Proposition~\ref{P:DM_11}, Corollaries~\ref{Cor:DM_10_Posit_Oper} and~\ref{C:DM_14}
imply the following statements:

(a) The Friedrichs extension $\wt{L}_F$ is given by the Dirichlet conditions: $y(0,x) = y(\pi, x)=0$, and the Krein extension is given by the boundary condition $\Gamma_1y = M(0)\Gamma_0y$
(note that the lower bound is $m(L_0)=1$). In view of the definition
of boundary triplet~\eqref{DM_Boun_Map}, the latter can be transformed as
$$
Y' + \begin{pmatrix}
\Lambda({\tan}\Lambda\pi)^{-1} \quad \Lambda(\shh\Lambda\pi)^{-1} \\
\Lambda(\shh\Lambda\pi)^{-1} \quad \Lambda({\tan}\Lambda\pi)^{-1}
 \end{pmatrix}Y = 0.
$$
Here $\Lambda = \sqrt {-d^2/dx^2}$  is the Calderon operator, i.e.,
the pseudo-differential operator in $L_2(-\infty, +\infty)$ with the symbol $|\xi|$.

(b) The negative part of the spectrum of the operator $\wt{L}_{\sigma_1}$ has dimension
$(0 \le )n \le \infty$ if and only if the same holds for the operator $B_1 - M(0)$, or,
equivalently, for the operator
$$
\wh{A}^{-1/4}(B_1 - M(0))\wh{A}^{-1/4} =\Sigma_1 +\begin{pmatrix}
\Lambda(\mbox{th }\Lambda\pi)^{-1} \quad \Lambda(\shh\Lambda\pi)^{-1} \\
\Lambda(\shh\Lambda\pi)^{-1} \quad \Lambda(\mbox{th }\Lambda\pi)^{-1}
\end{pmatrix},
$$
where $B_1$ is defined by~\eqref{8.4} and
$$
\Sigma_j := \begin{pmatrix} \sigma_{j,0}(x) \quad 0 \\ 0 \quad \sigma_{j,\pi}(x)
\end{pmatrix}\quad(j=1,2).
$$

(c) The operators $\wt{L}_{\sigma_1}$ and $\wt{L}_{\sigma_2}$
are resolvently comparable if $\Sigma_1(x) - \Sigma_2(x)\in L_1(\mathbb R)$, because
\[
\wh{A}^{-1/2}|\Sigma_1-\Sigma_2|^{1/2} \in \frak S_2(L_2(\mathbb R)\times L_2(\mathbb
R)).
\]
\end{example}

8.3. Further, let $L= L_{\min}$ be minimal operator generated in $L_2([0, \infty);H)$
by the differential expression
\[
l[y] = -y''+(A-I)y(t), \quad A \ge I, \quad t \in [0, \infty).
\]
As was shown in~\cite{Kut76}, the boundary triplet $\{\mathcal{H}, \Gamma_0, \Gamma_1\}$
for $L^*$ is of the form
  \begin{equation}\label{DM71A}
\mathcal{H} = H,\quad \Gamma_0y=A^{-1/4}y(0),\quad \Gamma_1y =
A^{1/4}(y'(0)+A^{1/2}y(0).
      \end{equation}
The defect subspace $\mathfrak{N}_z$ consists of vector functions
$\exp(-\sqrt{A+z}t)A^{1/4}f$, $f \in H$.
On this basis, it is easy to calculate the Weyl function:
\eq{DM_71}
{M(z) = A^{1/2}(A^{1/2}-(A-I-z)^{1/2}).}
With the specific form $\eqref{DM_71}$ of the Weyl function, by using the results
of sections~\ref{2} -- \ref{4} one can formulate different statements on the spectrum
of extensions in terms of the operator $A$ as well as on their resolvent
comparability, etc. We present just one of them.
  \begin{proposition}\label{P:DM_18}
Suppose that $S \in \mathcal{C}(H)$ and $\wt{L}_S= \wt{L}_S^*$ is the extension of
the operator $L_0$ generated by $y'(0) = Sy(0)$. For the negative part of the spectrum of the operator $\wt{L}_S$
\begin{itemize}
\item[\;\;\rm (a)] to consist of $0 \le n \le \infty$ points counting multiplicities;
\item[\;\;\rm (b)] to have the origin as a unique limit point,

it is necessary and sufficient that the same be valid for the operator
$S+(A-I)^{1/2}$.
\end{itemize}
\end{proposition}
   \begin{proof}
In terms of boundary triplet \eqref{DM71A}, the extension $\wt L_S$ is given by
$\dom(\wt L_S) = \ker(\Gamma_1- B \Gamma_0)$, where $B =
A^{1/4}(A^{1/2}+S)A^{1/4}$. In this case, there holds the equivalence:
\[
\wt L_S = {\wt L}_S^* \Longleftrightarrow B=B^*.
\]
From \eqref{DM_71} it follows the equality
  \begin{equation}\label{DM8.9}
M(0) = A^{1/4}(A^{1/2} - (A-I)^{1/2}).
  \end{equation}
In turn, this implies the relation $B-M(0) = A^{1/4}[S+(A-I)^{1/2}]A^{1/4}$.
The proof now follows from both Proposition~\ref{P:DM_11} and
the bounded invertibility of the operator $A \ge I$.
   \end{proof}

         \begin{remark}\label{R:DM_17}
Also it is interesting to note that the Friedrichs extension $\wt{L}_{F}$, as usual,
corresponds to the Dirichlet problem $y(0) = 0$, and, by Corollary~\ref{Cor:DM_10_Posit_Oper},
the Krein extension $\wt{L}_{K}$ is given by condition \eq{DM_72}{y'(0)=-(A-I)^{1/2}y(0),} in which the condition $\dom(\wt{L}_{K})=\ker(\Gamma_1 - M(0)\Gamma_0)$ has been transformed. Further,
the operators $\wt{L}_{F}$ and $\wt{L}_{K}$ are transversal since
\[
M(0)=A^{1/2}(A^{1/2}+(A-I)^{1/2})^{-1} \in [H],
\]
and, by Corollary~\ref{Cor:DM_10_Posit_Oper}, the transversality
of extensions $\wt{L}_{F}$ and $\wt{L}_{K}$ is equivalent to
the condition $M(0) \in [H]$.
     \end{remark}

    \begin{example}\label{Ex:DM_2}
Let $L = L_{\min}$ be minimal operator generated in $L_2(\mathbb{R}_+ \times
\mathbb{R})$ by the Laplace expression
\[
-\Delta=-\left(\frac{\partial^2 }{\partial t^2} + \frac{\partial^2
}{\partial x^2} \right),
\]
and let $\wt{L}_j$ be its extension given by the boundary condition
$$
\left[ \frac{\partial y(t,x)}{\partial t} - \sigma_j(x)y(t,x) \right]\mid_{t=0}, \quad
\sigma_j(x) \in L_{\infty}(\mathbb{R}), \quad j =1,2.
$$
In this case, Proposition~\ref{P:DM_18} in which one should put $A=-\frac{d^2}{dx^2}+I$ in $L_2(\mathbb{R})$, suggests that:
\begin{enumerate}
    \item [(a)] the Friedrichs extension $\wt{L}_{F}$
corresponds to the Dirichlet problem $y(0,x)=0$, and the Krein extension $\wt{L}_{K}$
as follows from~\eqref{DM_72} is given by
$$
[\partial y(t,x)/\partial t + \Lambda y(t,x)]\mid_{t=0} =0,
$$
where $\Lambda = \sqrt {-d^2/dx^2}$ is the Calderon operator.

    \item [(b)] The extensions $\wt{L}_{F}$ and $\wt{L}_{K}$
are transversal since, in view of~\eqref{DM8.9},
\[
M(0) =
\left(I-\frac{d^2}{dx^2}\right)^{1/2}\left[\left(I-\frac{d^2}{dx^2}\right)^{1/2}-\Lambda\right]^{-1}
\in [H] =  [L_2(\mathbb{R})].
\]
   \item [(c)] The negative part of the spectrum of the extension
$\wt L_1$ consists of $(0 \le )n \le \infty$ points if and only if
the operator $S+\Lambda$ satisfies the same property in $L_2(\mathbb{R})$,
where $(Sf)(x)=\sigma(x)f(x)$.

\item [(d)] If $\sigma_1 - \sigma_2 \in L_1(-\infty, +\infty)$, then the extensions
$\wt{L}_1$ and $\wt{L}_2$ are resolvently comparable.
\end{enumerate}
\end{example}

8.4. Suppose that $L = L_{\min}$ is minimal operator generated in $L_2([0,b],
\mathfrak{H})$ by the differential expression of the hyperbolic type
\[
l[y] = y''+Ay+q(t)y,\quad A \ge I.
\]
Boundary triplets $\{\mathcal{H}, \Gamma_0, \Gamma_1\}$ for $L^*$
are constructed in~\cite{VG75} and have the form:
\[
 \mathcal{H} = \mathfrak{H} \oplus \mathfrak{H},\quad \Gamma_0y =
\left(%
\begin{array}{c}
  y_0 \\
  y_b \\
\end{array}%
\right), \quad
\Gamma_1y=\left(%
\begin{array}{c}
  y'_0 \\
  y'_b \\
\end{array}%
\right),\quad\mbox{where}
\]
\[
\begin{split}
y_0&=2^{-1/2}(-\sin(\sqrt{A}b)A^{-1/2}y'(b)+\cos(\sqrt{A}b)y(b)+y(0)),\\
y_b&=2^{-1/2}(-\cos(\sqrt{A}b)A^{-1/2}y'(b)-\sin(\sqrt{A}b)y(b)+A^{-1/2}y'(0)),\\
y'_0&=2^{-1/2}(-\cos(\sqrt{A}b)y'(b)+A^{-1/2}\sin(\sqrt{A}b)y(b)-y'0)),\\
y'_b&=2^{-1/2}(-\sin(\sqrt{A}b)y'(b)+A^{-1/2}\cos(\sqrt{A}b)y(b)-A^{-1/2}y(0)).
\end{split}
\]
For any $ y(t) \in \mathfrak{N}_z$ there holds the representation
$$y(t,z) = \omega_1(t,z)f_1+\omega_2(t,z)f_2, \quad f_1, f_2 \in \mathfrak{H}.$$

Consequently, the Weyl function is of the form
\[
M(z) = \Omega_1(z)\Omega_0(z)^{-1},
\]
where
\[
\Omega_0(z)=  \left(%
\begin{array}{cc}
\Gamma_0  \omega _1(t,z) & \Gamma_0  \omega _2(t,z)
\end{array}%
\right), \quad
\Omega_1(z) = \left(%
\begin{array}{cc}
\Gamma _1 \omega _1(t,z) & \Gamma _1 \omega _2(t,z)
\end{array}%
\right).
\]
If $q(t) \equiv 0$, then
$$
\omega_1(t,z) = (t-b)\cos\sqrt{A-z},\quad \omega_2(t,z) =
(t-b)\frac{\sin\sqrt{A-z}}{\sqrt{A-z}}A,
$$
 and the Weyl function is calculated quite clearly.

\section{Schr\"odinger Operator in $\R^3\backslash\{0\}$}\label{8}

Consider in $L^2(\R^3)$ the Schr\"odinger operator
\[
l[y]=-\Delta y+q(x)y
\]
with a spherically symmetric potential $q(x)=q(|x|)$, $|q(x)|\leqslant C$ defined originally
on $C_0^\infty(\R^3\backslash0)$. Its closure $L = L_{\min}$ is a minimal
symmetric operator with the deficiency indices $(1,1)$. Starting with~\cite{BerFad61},
the operator $L$ has been studied by many authors (see the reference list in~\cite{Koc82}).
Boundary triplets for $L^*$ and more general elliptic operators are constructed
in~\cite{Koc82}.

For $q(x)\equiv 0$ the defect subspaces of the operator $L =-\Delta$ have the form
\[
\sN_z=\{e^{ir\sqrt{z}}/r\},\quad(r=|x|).
\]
Therefore, for any bounded $q(x)$ and for any $f\in\dom(L^*)$ there holds the relation
\begin{equation}\label{eq:Shred_f}
    f(x)=\frac{c_{-1}}{r}+c_0+\wt{f}(x),\quad
\wt{f}(x)\in\dom(L_0),\quad \wt{f}(0)=0.
\end{equation}
Applying the Green formula to a function $f\in\dom(L^*)$ of form~\eqref{eq:Shred_f}
and a function $g\in\dom(L^*)$ of the form
\[
g(x)=\frac{d_{-1}}{r}+d_0+\wt{g}(x),\quad \wt{g}(x)\in\dom(L),\quad \wt{g}(0)=0
\]
in the domain $G_r$ exterior to the sphere $\sum_r=\{x:|x|=r\}$, we get
\begin{eqnarray*}
(L^*f,g)-(f,L^*g)&=&\lim\limits_{r\to0}\iiint\limits_{G_r}(f\cdot
\overline{\Delta g}-\Delta f\cdot\bar{g})dx=\\
&{=}&\lim\limits_{r\to0}\iint\limits_{\sum_r}\left[f\left(\frac{\partial\bar{g}}
{\partial n}\right)-\left(\frac{\partial f} {\partial
n}\right)\bar{g}\right]d\sigma\\
&=&\lim\limits_{r\to0}\iint\limits_{\sum_r}
\left(\frac{c_0\bar{d}_{-1}-c_{-1}\bar{d_0}}{r^2}+o(1)\right)d\sigma=4\pi(c_0\bar{d}_{-1}-c_{-1}\bar{d_0}).
\end{eqnarray*}
Define the boundary triplet by setting
\[
\cH=\dC,\quad
\Gamma_0y=2c_{-1}\sqrt{\pi},\quad\Gamma_1y=2c_0\sqrt{\pi}.
\]
  \begin{proposition}\label{P:DM_19}
The Weyl function of the operator $L$ coincides with one of the Weyl functions
of the Sturm -- Liouville operator $A=- d^2/dx^2 + q(r)$ in $L^2[0,\infty)$
with the boundary conditions $y(0)=y'(0)=0$.
\end{proposition}
\begin{proof}
If $y(r,z)\in\sN_z(A)$, then $y(r,z)r^{-1}\in\sN_z(L_0^*)$. Let
\[
\frac{y(r,z)}{r}=\frac{c_{-1}(z)}{r}+c_0(z)+o(1).
\]
Then the Weyl function of the operator $L$ is of the form
$M_{L}(z)=c_0(z)/c_{-1}(z)$. Choosing for the operator$A^*$ the same boundary triplet as
in Example~\ref{ex:3.6} we obtain $M_{L}(z)=M_A(z)$.
\end{proof}

Consider the extension $\wt{L}_h$ of the operator $L_0$ defined by the boundary condition
$\Gamma_1y=h\Gamma_0y$, $h\neq\bar{h}$. It was shown by Pavlov~\cite{Pavlov66}
that, under the condition on the potential
\begin{equation}\label{DM_73}
\sup|q(r)|\exp(\varepsilon\sqrt{r})<\infty,
\end{equation}
the spectrum of the operator $\wt{A}_h$ and hence of the operator $\wt{L}_h$
has finitely many eigenvalues and spectral singularities.

In the same paper~\cite{Pavlov66} it was shown that this condition
is precise for the Sturm -- Liouville operator. Namely, there were presented both a
real  potential $q(x)$ for which condition~\eqref{DM_73} is violated but
$\sup|q(r)|\exp(\varepsilon r^\beta)<\infty$ for $0<\beta<1/2$, and
a complex $h\neq\bar{h}$ such that the operator has an infinite set of
eigenvalues. Proposition~\ref{P:DM_19} yields that, for the operator $L_0$
with a potential $\wt{q}(x)=q(|x|)$,  there also exists an extension
$\wt{L}_h$ with the same property.


\section{Laplace Operator in Domains with Piecewise Smooth Boundary }\label{9}

\subsection{Domain with One Incoming Angle} \label{Subsec:1}

Let $\Omega_\beta=\{(r,\varphi):0\leqslant r\leqslant
1,\,0\leqslant\varphi\leqslant\pi/\beta\}$ be a domain in $\R^2$,
and let $1/2<\beta<1$. The Laplace operator $Lu=-\Delta
u$ considered in $L^2(\Omega_\beta)$ with the Dirichlet conditions on the boundary
is a symmetric operator with the deficiency indices $(1,1)$
and the domain $\dom(L)=W_0^{2,2}(\Omega)$ (see~\cite{BirS62}).

The domain of the Friedrichs extension $L_F$ is of the form
\[
\dom(L_F)=W_2^{2,2}(\Omega_\beta)+\{u_\beta\},
\]
where
\[
u_\beta(x)=\eta_\varepsilon(r)r^\beta\sin\beta\varphi\in
W_0^{1,2}(\Omega),\quad    \Delta u_\beta\in L^2(\Omega).
\]
Here $\eta_\varepsilon(r)$ is a smooth "cutoff"\ function equal to $1$ for
$r\leqslant\varepsilon/2$ and to $0$ for $r\geqslant\varepsilon$.
The function\footnote[1]{The function $u_\beta(x)$ is constructed by Guseva
(see~\cite{BirS62}).} $u_\beta(x)$ is a weak solution to the problem
\[
-\Delta u=f,\quad u|_\Gamma=0\quad(f\in L^2(\Omega_\beta))
\]
which does not belong to $W_0^{2,2}(\Omega_\beta)$. For any
$f\in\dom(L^*)$ there holds the representation
\begin{equation}\label{DM_74}
f=f_0+c_1u_\beta+c_2v,
\end{equation}
where
\[
v=(r^{-\beta}-r^\beta)\sin\beta\varphi\in\sN_0(L)=\ker(L^*),\quad
c_1,c_2\in\dC.
\]
Let $g\in\dom(L^*)$. In view of~\eqref{DM_74}, the decomposition
\[
g=g_0+d_1u_\beta+d_2v,\quad d_1,d_2\in\dC
\]
holds true. Since
\[
(L^*f,g)-(f,L^*g)=c_1\bar{d}_2(L^*u_\beta,v)-c_2\bar{d}_1(v,L^*u_\beta),
\]
one can define a boundary triplet for the operator $L^*$
by setting
$$\begin{array}{cc}
  \Gamma_0f=kc_2=k\lim\limits_{r\to0}r^\beta f(re^{i\pi/2\beta}),  \\
  \Gamma_1f=k(c_1-c_2)=k\lim\limits_{r\to0}r^{-\beta}
  \{f(re^{i\pi/2\beta})-r^{-\beta}\Gamma_0f\},
\end{array}$$
where $k^2=(L^*u_\beta,v)$. It can be shown that $(L^*u_\beta, v)>0$.
Indeed,
\begin{eqnarray*}
L^*u_\beta&=&-\Delta u_\beta=-\left(\frac{\partial^2}{\partial
r^2}+\frac{1}{r}\frac{\partial}{\partial
r}+\frac{1}{r^2}\frac{\partial^2}{\partial\varphi^2}\right)u_\beta=\\
&=&-\left((2\beta+1)
r^{\beta-1}\eta_\varepsilon'(r)+\eta_\varepsilon''(r)r^\beta
-\beta\eta_\varepsilon(r)r^{\beta-2}\right)\sin\beta\varphi.
\end{eqnarray*}
It follows that
\begin{eqnarray*}
(L^*u_\beta,v)&=&\iint\limits_{\Omega_\beta}\left((2\beta+1)
r^{\beta-1}\eta_\varepsilon'(r)+\eta_\varepsilon''(r)r^\beta
-\beta\eta_\varepsilon(r)r^{\beta-2}\right)(r^\beta-r^{-\beta})
\sin\beta\varphi dxdy=\\
&=&\int\limits_{\varepsilon/2}^\varepsilon\left((2\beta+1)\eta_\varepsilon'(r)+
\eta_\varepsilon''(r)r-\beta\eta_\varepsilon(r)r^{-1}\right)(r^{2\beta}-1)dr
\int\limits_0^{\pi/\beta}\sin\beta\varphi d\varphi.
\end{eqnarray*}
By integrating by parts we obtain:
\[
\int\limits_{\varepsilon/2}^\varepsilon\eta_\varepsilon''(r)(r^{2\beta+1}-r)dr=
\int\limits_{\varepsilon/2}^\varepsilon\eta_\varepsilon'(r)[1-(2\beta+1)r^{2\beta}]dr.
\]
A comparison of the last two equalities leads to the relation
\[
(L^*u_\beta,v)=2\int\limits_{\varepsilon/2}^\varepsilon
\eta_\varepsilon(r)(1-r^{2\beta-1})dr>0.
\]

Let us find the Weyl function of the operator $L$. If $f(\cdot,{z})\in\sN_{z}$, then
\[
f(x,{z})=\left[m({z})J_\beta(r\sqrt{{z}})
+J_{-\beta}(r\sqrt{{z}})\right] \sin\beta\varphi,\qquad
m({z})=-\frac{J_{-\beta}(\sqrt{{z}})}{J_\beta(\sqrt{{z}})},
\]
where $J_{\pm\beta}(r)$ are the Bessel functions of the first kind. It follows that
\begin{equation}
  \Gamma_0f=k\frac{1}{\Gamma(1-\beta)}(2^{-1}\sqrt{{z}})^{-\beta},\qquad
   \Gamma_1f=k\frac{m({z})}{\Gamma(1+\beta)}(2^{-1}\sqrt{{z}})^{\beta},
\end{equation}
\begin{equation}\label{DM_76}
  M({z})=-\frac{\Gamma(1-\beta)J_{-\beta}(\sqrt{{z}})(2^{-1}\sqrt{{z}})^{2\beta}}
  {\Gamma(1+\beta)J_\beta(\sqrt{{z}})}.
\end{equation}

Any extension $\wt{L}_h$ of the operator $L$ can be defined by the condition
$\Gamma_1f=h\Gamma_0f$. Studying the behavior of zeros of the function $M({z})-h$
we can prove completeness and basis properties of sets  of eigenvectors and associated vectors of the operator $\wt{L}_h$.
\subsection{Domain with Finite Number of Incoming Angles}\label{Subsec:2}

Later on, let $\Omega$ be a bounded domain in $\R^2$ with a
piecewise smooth boundary of class $C^2$, and let $a_k$ 
be corner points of the boundary $\Gamma=\partial\Omega$
in which the interior angle $\pi/\beta_k$ is greater than
$\pi$ $(1/2<\beta_k<1,\ k=1,\ldots,n)$. Then the Laplace operator
$Lu=-\Delta u$ considered in $L^2(\Omega)$ with a Dirichlet condition
on the boundary is a symmetric operator with the deficiency indices $(n,n)$,
$\dom(L)=W^{2,2}_0$ (see~\cite{BirS62}).

Let a domain $\Omega$ be such that, for some collection of neighborhoods
$O(a_j,\varepsilon_j)$ of the corner points $a_j$, a part of the boundary
$\Gamma_j=\Gamma\cap O(a_j,\varepsilon_j)$ is composed of two
straight line segments
\[
\arg(x-a_j)=\theta_j,\quad \arg(x-a_j)=\theta_j+\frac{\pi}{\beta_j}.
\]
Consider the mapping
$G_j(x)=[e^{-i\theta_j}(x-a_j)]^{\beta_j}$ that takes
$\omega_j=O(a_j,\varepsilon_j)\cap\Omega$ to the upper half-plane
$\dC_+$ such that $\Gamma\cap O(a_j,\varepsilon_j)$ goes to a real
line segment. Let $F_j(x)$ be a function that takes $G_j(\Omega)$
onto $\dC_+$ such that $F_j(0)=0$, $F_j'(0)=1$. By the principle of symmetry,
the function $F_j(x)$ admits analytic continuation to some neighborhood
of $x=0$ and hence it can be represented by the Taylor series
\begin{equation}\label{DM_77}
F_j(x)=x+d_{j,2}x^2+d_{j,3}x^3+\ldots+d_{j,k}x^k+\ldots
\end{equation}
with real coefficients.

Assume that all the functions $F_j(x)$ $(1\leqslant j\leqslant n)$ are analytic in
the disk $|x|<\varepsilon$, and construct a finite smooth function
$\eta_\varepsilon(r)$ equal to $1$ for $r\leqslant\varepsilon/2$ and $0$ for
$r\geqslant\varepsilon$. Put
\begin{equation}\label{DM_78}
u_{\varepsilon, j}=\Im F_j(G_j(x))\cdot\eta_\varepsilon(|G_j(x)|),
\end{equation}
\begin{equation}\label{DM_79}
v_{j}=-\Im [F_j(G_j(x))]^{-1}.
\end{equation}

As was shown in~\cite{BirS62}, the domain of the Friedrichs extension $L_F$
of the operator $L$ consists of functions of the form
\begin{equation}\label{DM_80}
{u}(x,y) =u_0(x,y)+\sum\limits_{j=1}^n c_{1,j}(u)u_{\varepsilon,
j}(x),
\end{equation}
where $u_0(x,y)\in W_0^{2,2}(\Omega)$. It is easy to see that in
the polar coordinate system with the pole at $a_j$ and the polar ray
$\arg(x-a_j)=\theta_j$ the functions $u_{\varepsilon,j}(x)$ have the form
\[
u_{\varepsilon,j}(x)=
r^{\beta_j}\sin\beta_j\varphi\cdot\eta_\varepsilon(r,\varphi)\in
W_0^{2,1}(\Omega),
\]
the functions satisfy $v_j(x)\in\ker L^*$ (see~\cite[p. 19]{BirS62}), and
formulas~\eqref{DM_77}, \eqref{DM_79} imply that
\begin{equation}\label{DM_81}
    \begin{split}
v_j(x)=&-\Im\{[e^{-i\theta_j}(x-a_j)]^{-\beta_j}
+(d_{1,j}^2-d_{2,j})[e^{-i\theta_j}(x-a_j)] ^\beta_j+\\ \nonumber
&+o(|x-a_j|^{\beta_j})=(r^{-\beta_j}-d_jr^{\beta_j})\sin\beta_j\varphi+o(r^{\beta_j}),
\end{split}
\end{equation}
where $d_j=d_{1,j}^2-d_{2,j}$. As in~\ref{Subsec:1}, one can show that
\[
k_j^2:=(L^*u_{\varepsilon ,j},v_j)>0
\]
and it does not depend on $\varepsilon>0$ since $u_{\varepsilon_1,j}-u_{\varepsilon_2,j}\in\dom(L)$. Any function
$v\in\ker L^*$ can be represented as
\begin{equation}\label{DM_82}
v=\sum\limits_{j=1}^n c_{2,j}(v)v_j.
\end{equation}
Note that $(L^* u_{\varepsilon, k},v_j)=0$ for $k\neq j$. Indeed, assume that
\[
\varepsilon<2^{-1}\min|a_k-a_j|,\quad \eta(x)=1\quad\mbox{ for
}\quad x\in O(a_k,\varepsilon)
\]
and $\supp\eta(x)\subset O(a_k,r_\varepsilon)$. Then
$\eta(x)v_j\in\dom(L)$ and $L(\eta(x)v_j)=0$ for any $x\in
O(a_k,\varepsilon)$,
\[
(L^*u_{\varepsilon, k},v_j)=(L^*u_{\varepsilon,
k},\eta(x)v_j)=(u_{\varepsilon, k},L(\eta(x)v_j))=0.
\]
Define the boundary triplet $\{\cH, \Gamma_0,\Gamma_1\}$ by setting
\[
\cH=\dC^n,\quad (\alpha,\beta)_\cH=\sum\limits_{j=1}^n
k_j\alpha_j\bar{\beta}_j,\quad \Gamma_l
u=(\Gamma_{l,1}u,\ldots,\Gamma_{l,n}u)^T\quad (l=0,1),
\]
where
\begin{equation}\label{DM_83}
\Gamma_{0,j}=\widehat{\lim\limits_{x\to a_j}}|x-a_j|^{\beta_j}u(x),
\end{equation}
\begin{equation}\label{DM_84}
\Gamma_{1,j}=\widehat{\lim\limits_{x\to
a_j}}|x-a_j|^{-\beta_j}\{u(x)-\Gamma_{0,j}(u)|x-a_j|^{\beta_j}\}.
\end{equation}
The symbol $\widehat{\lim\limits_{x\to a_j}}$ will denote the limit of $u(x)$
along the bisector of the interior angle of the domain $\Omega$ with vertex at $a_j$.

Let $u_l\in\dom(L^*)$, $l=0,1$. Then, by taking into account
\eqref{DM_81} -- \eqref{DM_84}, we have
\[
u_l=u_l^0+\sum\limits_{1\leqslant j\leqslant
n}\left[\Gamma_{1,j}u_l+d_j\Gamma_{0,j}u_l\right]u_{\varepsilon, j}+
\sum\limits_{1\leqslant j\leqslant n}\Gamma_{0,j}(u_l)v_j,
\]
\[
u_l^0\in W^{2,0}_0(\Omega),\qquad l=0,1.
\]
By direct substitution we verify the equality
\[
(L^*
u_1,u_0)-(u_1,L^*u_0)=(\Gamma_1u_1,\Gamma_0u_0)_{\cH}-(\Gamma_0u_1,\Gamma_1u_0)_\cH.
\]

In the defect subspace $\sN_{z}$, chhose the basis consisting of functions
$u_l(x,{z})$, $(x\in\Omega,\,\, l=1,2,\ldots,n)$ such that
$u_l(x,{z})\in W_0^{2,2}(\omega_j)$ for $j\neq l$. Then the Weyl function
$M({z})$ in this basis has the diagonal representation
\begin{equation}\label{DM_85}
M({z})=\|m_{j,l}({z})\|_{j,l=1}^n, \quad\mbox{ where }\quad
m_{j,l}({z})=\frac{\Gamma_{1,j}(u_l(x,{z}))}{\Gamma_{0,j}(u_l(x,{z}))}\delta_{j,l}.
\end{equation}

\subsection{Unbounded Domain with One Incoming Angle}\label{Subsec:3}

Consider the operator $Lu=-\Delta u$ defined in the unbounded domain
\[
\Omega=\{(r,\varphi):r\geqslant0,0\leqslant\varphi\leqslant\pi/\beta\},\quad
(2^{-1}<\beta<1)
\]
with the Dirichlet conditions on the boundary. The operator $L$ is symmetric,
with the deficiency indices $(1,1)$, and its simple part is unitary equivalent
to the operator $A$ generated by the differential expression
\[
l[y]=-\frac{d^2y(r)}{dr^2}+\frac{\beta^2-1/4}{r^{2}}y(r)
\]
in $L^2(0,\infty)$. The defect space $\sN_{z}$ of the operator $A$
consists of functions of the form
\[
y_{z}(r)=r^{1/2}H_\beta^{(1)}(r\sqrt{z}),
\]
where
$H_\beta^{(1)}(r)=J_\beta(r)+iY_\beta(r)$ is the Hankel function.

Define the boundary triplet $\{\cH, \Gamma_0,\Gamma_1\}$ by setting
$\cH=\dC$,
\[
\Gamma_0y=\lim\limits_{r\to0}r^{\beta-1/2}y(r),\qquad
\Gamma_1y=\lim\limits_{r\to0}r^{-\beta-1/2}\left[y(x)-\Gamma_0(y)r^{1/2-\beta}\right].
\]
From the asymptotics of the Hankel function as $r\to0$
\[
H_\beta^{(1)}(r\sqrt{{z}})=\frac{i}{\sin\beta\pi}\left\{e^{-i\beta\pi}
\left(\frac{r\sqrt{{z}}}{2}\right)^\beta\frac{1}{\Gamma(1+\beta)}-
\left(\frac{r\sqrt{{z}}}{2}\right)^{-\beta}\frac{1}{\Gamma(1-\beta)}
+o(r^\beta)\right\},
\]
we easily find that
\[
\begin{array}{l}
  \Gamma_0y_{z}=-{\displaystyle\frac{i}{\sin\beta\pi}\left(\frac{\sqrt{{z}}}{2}
  \right)^{-\beta}\cdot\frac{1}{\Gamma(1-\beta)}}, \\
  \Gamma_1y_{z}={\displaystyle\frac{i}{\sin\beta\pi}\left(\frac{\sqrt{{z}}}{2}
  \right)^{\beta}\cdot\frac{e^{-i\beta\pi}}{\Gamma(1+\beta)}},
\end{array}
\]
and the following expression for the Weyl function
\[
M({z})=C_\beta{z}^\beta,
\]
where
$C_\beta=\exp(-i\beta\pi)4^{-\beta}\Gamma(1-\beta)/\Gamma(1+\beta)$.

The Friedrichs extension $\wt{L}_F$ is given by the condition
$\dom(\wt{L}_F)=\ker\Gamma_0$, and the Krein extension is given by the condition
$\dom(\wt{L}_K)=\ker\Gamma_1$, because $M(0)=0$. The characteristic function
$\Theta({z})$ of the extension $L_h$
$(\dom(\wt{L}_h)=\ker(\Gamma_1-C_\beta h\Gamma_0))$ is of the form
\[
\Theta({z})=\frac{{z}^\beta+h}{{z}^\beta+\bar{h}}.
\]


\end{document}